\documentclass[11pt,letterpaper]{article}

\usepackage[T1]{fontenc}
\usepackage{microtype}
\usepackage[margin=1in]{geometry}

\usepackage{mathtools}
\usepackage{amssymb}
\usepackage{amsthm}
\usepackage{algorithm}
\usepackage[noend]{algpseudocode}
\usepackage{array}
\usepackage{multirow}
\usepackage{booktabs}
\usepackage{longtable}
\usepackage{float}
\usepackage{tikz}
\usetikzlibrary{decorations.pathreplacing}

\usepackage{natbib}
\usepackage{hyperref}
\usepackage[nameinlink,noabbrev]{cleveref}
\usepackage{xcolor}
\usepackage{footnotehyper}
\usepackage{etoolbox}

\makeatletter
\providecommand*{\theHALG@line}{}
\renewcommand*{\theHALG@line}{\thealgorithm.\arabic{ALG@line}}
\makeatother

\allowdisplaybreaks
\numberwithin{equation}{section}

\definecolor{refbox}{RGB}{50,240,50}
\definecolor{citebox}{RGB}{50,240,50}

\hypersetup{
    colorlinks=false,
    pdfborder={0 0 0.6},
    linkbordercolor=refbox,
    citebordercolor=citebox
}

\theoremstyle{plain}
\newtheorem{theorem}{Theorem}

\newtheorem{lemma}{Lemma}

\theoremstyle{definition}

\theoremstyle{remark}

\theoremstyle{plain}
\newtheorem{remark}{Remark}

\newcommand{\NN}{\mathbb{N}}
\newcommand{\RR}{\mathbb{R}}

\newcommand{\EE}{\mathbb{E}}
\newcommand{\PP}{\mathbb{P}}

\newcommand{\cA}{\mathcal{A}}
\newcommand{\cC}{\mathcal{C}}
\newcommand{\cD}{\mathcal{D}}
\newcommand{\cI}{\mathcal{I}}
\newcommand{\cL}{\mathcal{L}}
\newcommand{\cO}{\mathcal{O}}
\newcommand{\cP}{\mathcal{P}}
\newcommand{\cR}{\mathcal{R}}
\newcommand{\cS}{\mathcal{S}}

\DeclareMathOperator*{\argmax}{arg\,max}
\DeclareMathOperator{\OPT}{OPT}
\DeclareMathOperator{\ALG}{ALG}
\DeclareMathOperator{\LP}{LP}
\DeclareMathOperator{\suchthat}{s.t.}

\newcommand{\one}{\mathbf{1}}

\newcommand{\ST}{\mathrm{ST}}
\newcommand{\FS}{\mathrm{FS}}
\newcommand{\MMNL}{\mathrm{MMNL}}
\newcommand{\D}{\mathrm{D}}
\newcommand{\C}{\mathrm{C}}
\newcommand{\M}{\mathrm{M}}
\newcommand{\AC}{\mathrm{AC}}
\newcommand{\BM}{\mathrm{BM}}

\usepackage{threeparttable}

\title{A Unified Framework for Assortment Optimization and Matching under MNL Choice}

\author{Omar El Housni \qquad Ulysse Hennebelle\\[0.4em]
  \normalsize\itshape School of Operations Research and Information Engineering,
  Cornell Tech, Cornell University\\
  \normalsize\itshape \{oe46,uh34\}@cornell.edu
}

\date{\today}

\begin{document}
\maketitle

\begin{abstract}
We develop a unified framework for assortment optimization
and matching under multinomial logit (MNL) choice, yielding
simple algorithms and analyses across a broad class of problems.
The framework centers on a single-customer assortment problem
in which the firm commits to eligible products before their
independent availabilities are revealed.
A simple threshold policy achieves a tight $1/2$ guarantee
against a benchmark that selects its assortment after observing
availability and may control dependence subject to the same
marginal probability bounds. This guarantee accommodates
cardinality constraints and adversarial removal of excess
available products. Combined with problem-specific linear
programming relaxations, the framework yields $1/2$ guarantees
for standard two-sided assortment optimization, customization
under mixtures of MNL models, ride-hailing notifications,
multi-round broadcast matching, and MNL prophet inequalities,
and a $1/4$ guarantee for fully static two-sided assortment
optimization. Across these applications, we improve
approximation factors, extend guarantees to more general
revenues and cardinality constraints, and retain existing
factors against stronger benchmarks.
\end{abstract}

\section{Introduction}\label{section:introduction}

An assortment is a set of alternatives made available for selection, and assortment optimization concerns choosing which alternatives to include to maximize an objective such as expected revenue. Such decisions arise, for example, in online retail, where a firm chooses which products to offer to customers. They also arise on choice-based matching platforms, where the platform presents potential partners to agents on both sides of a market, and agents choose among the alternatives offered to them. In both settings, a decision-maker controls the offered sets, while the choices made from those sets determine which transactions occur. A substantial literature studies these decisions under the multinomial logit (MNL) choice model, both in assortment optimization and in matching markets (\citealp{talluri04,rusmevichientong2010dynamic,davis2013assortment,rusmevichientong2014assortment,ashlagi_etal19,housni2024two}). In these problems, alternatives compete for selection, so their contributions to expected revenue cannot be evaluated in isolation. Offering an additional alternative creates a new opportunity for a transaction, but also diverts demand from existing alternatives.

Many problems in the literature on assortment optimization and choice-based matching involve uncertainty not only about the final choice, but also about which alternatives will be available for a transaction.
For example, in a sequential two-sided market, customers first choose suppliers, and each supplier then chooses among the customers who selected it \citep{ashlagi_etal19,torrico21,nissim2024revenue}.
The set of potential partners available to a supplier is therefore random and depends on earlier assortment decisions.
Related challenges arise in ride-hailing notifications \citep{ekbatani2026non}, where a platform decides whom to notify before observing acceptances, and in online assortment selection, where product characteristics are revealed sequentially \citep{goyal2023mnl}.
Although these models differ in their operational details, they share several features: selection decisions are made with incomplete information, the alternatives that can be offered depend on random outcomes, and feasibility constraints limit the sets that can ultimately be presented.
Their analysis must account for these sources of uncertainty and the choices made from the resulting sets.

Our goal is to develop a simple analytical and algorithmic tool that exploits this shared structure and provides a unified approach to a broad class of MNL-based optimization problems. We use this framework to simplify the analysis of existing models and design algorithms with stronger approximation guarantees. We build the tool around a single-customer assortment problem with random product availability and an advance commitment on which products can be offered. This formulation captures a common structure across distinct models: a product may represent a retail item or a potential matching partner, while availability may arise from earlier choices or from algorithmic randomization. By comparing the value of an advance commitment with that of selection decisions made after availability is revealed, we derive a general guarantee that serves as a reusable building block across many settings and models in the literature.

\subsection{Problem description}
Consider a firm offering products to a single customer from a universe of substitutable products $[n]=\{1,\ldots,n\}$, where each product $i$ has a nonnegative revenue $r_i$ and a nonnegative preference weight $v_i$.
The displayed assortment may contain at most $k$ products, where $k$ is a positive integer, with the convention $k=\infty$ when no display limit is imposed.
When offered an assortment $T\subseteq[n]$, the customer chooses
according to the MNL choice model: product $i\in T$ is selected
with probability $\phi_{\mathbf v}(i,T)=v_i/(1+\sum_{j\in T}v_j)$,
and the exit option is selected with the remaining probability.
The exit option has preference weight one and generates no revenue.
Thus, the firm's expected revenue from displaying $T$ is
\[
    R_{{\mathbf{r}},{\mathbf{v}}}(T)=\sum_{i\in T} r_i \phi_{\mathbf{v}}(i,T)=\frac{\sum_{i\in T}r_iv_i}
                 {1+\sum_{i\in T}v_i}.
\]
We simply write $R(T)$ when the dependency on the revenue vector $\mathbf{r}$ and the preference weight vector $\mathbf{v}$ is clear from the context.
We consider a setting where product $i$ is available for display with probability $x_i$. We denote the vector of availability probabilities as ${\mathbf{x}}=(x_1,\ldots,x_n)\in[0,1]^n$. We introduce the following three optimization problems.

\paragraph{Committed firm.}
Each product $i$ is independently available with probability $x_i$.
Let $A\subseteq[n]$ denote the random set of available products.
Before observing $A$, the firm chooses a set $S\subseteq[n]$ of products eligible to be displayed; the size of this initial set is unrestricted.
If $|S\cap A|\le k$, all products in $S\cap A$ are displayed.
Otherwise, an adversary selects the $k$ products in $S\cap A$ whose display minimizes the firm's expected revenue.
We call this rule \emph{adversarial truncation} and write
\[
    R_k^-(B)
    =
    \min_{\substack{
        T\subseteq B\\
        |T|=\min\{k,|B|\}
    }}
    R(T)
\]
for the revenue obtained by applying it to a set $B\subseteq[n]$.
Writing $\cI(\mathbf{x})$ for the distribution of $A$, the
committed firm's optimal expected revenue is
\begin{equation}\tag{\sf Committed}\label{eq:opt_c}
    \OPT_{\C}
    =
    \max_{S\subseteq[n]}
    \EE_{A\sim\cI(\mathbf{x})}\bigl[R_k^-(S\cap A)\bigr].
\end{equation}

\paragraph{Delayed firm.}
Like the committed firm, the delayed firm faces independent product availabilities, with each product $i$ available with probability $x_i$. The delayed firm observes the available set $A$ before choosing an assortment of at most $k$ products from this set. For any $B\subseteq[n]$, define
\[
    R_k^+(B)
    =
    \max_{\substack{T\subseteq B\\|T|\le k}}R(T),
\]
the maximum revenue obtainable by offering at most $k$ products from $B$. The delayed firm's optimal expected revenue is therefore
\begin{equation}\tag{\sf Delayed}\label{eq:opt_d}
    \OPT_{\D}
    =
    \EE_{A\sim\cI(\mathbf{x})}\bigl[R_k^+(A)\bigr].
\end{equation}
We refer to the ability to tailor the displayed assortment to the realized available set as \emph{customization}.

\paragraph{Magician benchmark.}
The magician provides a benchmark for both firms. Like the delayed firm, it may choose an assortment of at most $k$ products after observing the available set. In addition, it may choose the joint distribution of product availabilities, subject to an upper bound $x_i$ on each product $i$'s availability probability.
Specifically, let $\cP(\mathbf{x})$ denote the set of distributions $\alpha$ over subsets of $[n]$ satisfying
\[
    \PP_{A\sim\alpha}(i\in A)\le x_i
    \qquad\text{for every }i\in[n].
\]
The magician first chooses a distribution $\alpha\in\cP(\mathbf{x})$. The available set $A$ is then drawn from $\alpha$, and the magician selects its best feasible assortment from $A$. Its optimal expected revenue is
\begin{equation}\tag{\sf Magician}\label{eq:opt_m}
    \OPT_{\M}
    =
    \max_{\alpha\in\cP(\mathbf{x})}
    \EE_{A\sim\alpha}\bigl[R_k^+(A)\bigr].
\end{equation}

The customer's choice from the selected assortment still follows the same MNL model. By construction, the three optimal values satisfy
\[
    \OPT_{\C}
    \le
    \OPT_{\D}
    \le
    \OPT_{\M}.
\]
For each pair $(A,B)\in\{(\C,\D),(\C,\M),(\D,\M)\}$, we define the \emph{gap} between $A$ and $B$ as the supremum of $\OPT_B/\OPT_A$ over all problem instances. This gap measures the maximum multiplicative advantage of $B$ over $A$.

\subsection{Our contributions}\label{subsection:contributions}

Our main contribution is a simple, polynomial-time threshold policy that achieves a tight $1/2$ guarantee against the magician benchmark, even under adversarial truncation. The policy and its analysis provide a unified tool for a broad class of assortment optimization and choice-based matching problems under the MNL model. We use this tool to simplify the analysis of several existing models and develop algorithms that improve approximation guarantees, achieve guarantees against stronger benchmarks, and accommodate more general revenue structures and cardinality constraints.

\begin{theorem}\label{thm:mnl}
There is a polynomial-time algorithm that computes a committed assortment whose expected revenue is at least half of the magician's optimal expected revenue.
\end{theorem}

Theorem~\ref{thm:mnl} bounds the combined advantage of
customization and control over dependence by a factor of
two, even with cardinality constraints and adversarial
truncation. It therefore bounds all three gaps by two.
Appendix~\ref{appendix:benchmark_gaps} shows that these bounds are tight;
in particular, the $1/2$ guarantee is worst-case optimal
over the entire class of committed or delayed policies.

\paragraph{Threshold policy.}
To prove Theorem \ref{thm:mnl}, our algorithm utilizes a simple threshold structure. Given the cardinality constraint $k$, we compute, for each product $i$, a {\em distorted revenue} that depends on $v_i$ and $r_i$. The algorithm then commits to every product whose distorted revenue is at least a common threshold. 
In Section~\ref{section:main_result}, we present the algorithm and prove Theorem~\ref{thm:mnl}. In Appendix~\ref{appendix:benchmark_gaps}, we establish matching lower bounds for all three pairwise gaps.

\paragraph{\underline{Applications:}}
We apply our framework to a broad range of assortment optimization and matching problems from the literature, spanning static, multi-round, and online settings. Each application combines a problem-specific linear programming relaxation with the same threshold analysis, providing a common approach to algorithm design across these models. The resulting improvements include better approximation factors, guarantees against stronger benchmarks, extensions to more general revenue structures, and guarantees under cardinality constraints in settings where previous guarantees applied only to unconstrained assortments. Table~\ref{tab:application_improvements} summarizes the resulting guarantees and their comparison with prior work.

\begin{table}[!b]
    \centering
    \begin{threeparttable}
        \small
        \setlength{\tabcolsep}{4pt}
        \renewcommand{\arraystretch}{1.3}

        \begin{tabular}{
            @{}
            >{\raggedright\arraybackslash}m{0.35\textwidth}
            >{\centering\arraybackslash}m{0.24\textwidth}
            >{\centering\arraybackslash}m{0.23\textwidth}
            >{\centering\arraybackslash}m{0.13\textwidth}
            @{}
        }
            \toprule
            \textbf{Application}
            &
            \textbf{Unconstrained guarantee}
            &
            \textbf{Constrained guarantee}
            &
            \textbf{Our guarantee}
            \\
            \midrule

            Customized standard TSAO\tnote{1}
            &
            $\tfrac12-\varepsilon$\tnote{1}
            &
            ?
            &
            $\tfrac12$
            \\

            \addlinespace

            Inclusive standard TSAO\tnote{2}
            &
            $\tfrac12$\tnote{2}
            &
            ?
            &
            $\tfrac12$
            \\

            \addlinespace

            Fully static TSAO\tnote{3}
            &
            $0.067$\tnote{3}
            &
            $0.067$\tnote{3}
            &
            $\tfrac14$
            \\

            \addlinespace

            MMNL with customization\tnote{4}
            &
            $\tfrac12-\varepsilon$\tnote{5}
            &
            $\Omega(1/\sqrt{m\log m})$\tnote{4}
            &
            $\tfrac12$
            \\

            \addlinespace

            Ride-hailing notifications\tnote{6}
            &
            $\tfrac{1-\delta}{4}$\tnote{6}
            &
            $\tfrac{1-\delta}{4}$\tnote{6}
            &
            $\tfrac12$
            \\

            \addlinespace

            Multi-round broadcast matching\tnote{7}
            &
            $\tfrac{1-e^{-3}}{3}$\tnote{7}
            &
            $\tfrac{1-e^{-3}}{3}$\tnote{7}
            &
            $\tfrac12$
            \\

            \addlinespace

            MNL prophet inequality\tnote{8}
            &
            $\tfrac12$\tnote{8}
            &
            $\tfrac12$\tnote{8}
            &
            $\tfrac12$
            \\

            \bottomrule
        \end{tabular}

        \caption{Approximation guarantees across applications.
        TSAO stands for two-sided assortment optimization, and
        MMNL stands for mixture of MNL models.
        The application descriptions below specify the models,
        constraints, benchmarks, and any additive error terms
        underlying these guarantees.
        A question mark indicates that the cited work does not
        provide a guarantee for the corresponding constrained variant.}
        \label{tab:application_improvements}

        \begin{tablenotes}[para,flushleft]
            \footnotesize
            \item[1] \citep{ahmadnejadsaein2025adaptive}.
            \item[2] \citep{nissim2024revenue}.
            \item[3] \citep{housni2024two}.
            \item[4] \citep{el2023joint}.
            \item[5] \citep{udwani2023submodular}.
            \item[6] \citep{ekbatani2026non}.
            \item[7] \citep{mohammadimulti}.
            \item[8] \citep{goyal2023mnl}.
        \end{tablenotes}

    \end{threeparttable}
\end{table}

\paragraph{Two-sided assortment optimization.}
A platform offers customers and suppliers assortments of potential partners, and a match forms when two agents select each other.

\underline{\emph{Standard TSAO.}}
We use this term for models in which customers choose first and suppliers subsequently choose from their backlogs of selecting customers~\citep{ashlagi_etal19,torrico21}. Customer assortments may be fixed in advance (\emph{static}) or adapted to earlier customer choices (\emph{adaptive})~\citep{housni2024two}. An \emph{inclusive} policy displays each supplier's entire backlog, while a \emph{customized} policy may select a subset.
For unconstrained assortments under MNL choice with pair-dependent
revenues, \citet{nissim2024revenue} initially obtained a
$(10/539-\varepsilon)$-approximation guarantee for the 
static inclusive problem. In a recent revision of the same paper, concurrent with our work, they improved this guarantee to $1/2$. 
\citet{ahmadnejadsaein2025adaptive} obtain a static $(1/2-\varepsilon)$ approximation against the optimal adaptive customized policy. We obtain a polynomial-time $1/2$ guarantee against this benchmark while fixing customer assortments in advance, even under adversarial truncation of supplier backlogs (Theorem~\ref{thm:standard_tsao}). Our algorithm accommodates cardinality constraints on both sides and arbitrary pair-dependent revenues, providing the first constant-factor guarantee for this constrained setting through a simpler algorithm and analysis, while also improving the state-of-the-art approximation guarantees for the unconstrained setting.

\underline{\emph{Fully static TSAO.}}
Here, assortments on both sides are fixed before any choices are observed~\citep{ahmed2022parameterized,housni2024two}. Under MNL choices, \citet{housni2024two} obtain a $0.067$ approximation for uniform match revenues with cardinality constraints. We improve the guarantee to $1/4$ and allow arbitrary pair-dependent revenues while retaining cardinality constraints on both sides (Theorem~\ref{thm:fully_static}). Our algorithm also avoids the case distinctions based on preference weights used in the previous approach.

\paragraph{MMNL assortment optimization with customization.}
In the model of \citet{el2023joint}, a firm selects an initial assortment subject to a cardinality constraint before observing the arriving customer's class. Each of the $m$ classes has its own MNL choice model, and the firm may display a class-dependent subset of the initial assortment. For this problem, \citet{el2023joint} obtained an $\Omega(1/\log m)$ approximation, and \citet{udwani2023submodular} improved the guarantee to $(1/2-\varepsilon)$.
We obtain a $1/2$ approximation through a simple randomized polynomial-time algorithm. Our guarantee also extends to a model in which revenues depend on both the product and the customer class, and separate cardinality constraints apply to the initial assortment and the displayed assortments (Theorem~\ref{thm:mmnl}).

\paragraph{Ride-hailing notifications.}
In the non-exclusive ride-hailing notification model of \citet{ekbatani2026non}, a platform may notify multiple drivers about the same rider, while notifying each driver about at most one rider. Drivers accept their notifications independently. When multiple drivers accept the same rider, \emph{First Acceptance} assigns one of them uniformly at random, whereas \emph{Best Acceptance} selects the accepting driver with the highest match value. The authors obtain expected welfare at least
$(1-\delta)/4$ times the optimal value under First
Acceptance, up to an additive loss of
$\mathcal{O}(\varepsilon)$.  
We reduce this problem to standard TSAO using the \emph{ranked MNL} choice model. The reduction yields a simple polynomial-time algorithm that achieves at least half of the optimal value under Best Acceptance, even when evaluated under \emph{Worst Acceptance}, where an adversary selects the accepting driver with the lowest match value for each rider (Theorem~\ref{thm:notifications}).

\paragraph{Multi-round broadcast matching.}
In the model of \citet{mohammadimulti}, a delivery platform broadcasts orders to drivers over several rounds.
Each driver receives at most one offer per round, and each order--driver pair may be offered at most once.
Drivers accept independently, an order with multiple acceptances is assigned uniformly among its accepting drivers, and matched orders and drivers leave the system.
For the general multiple-order setting, \citet{mohammadimulti} obtain a $(1-e^{-3})/3$ approximation against the optimal adaptive policy. We improve the approximation factor to $1/2$ against the same adaptive benchmark (Theorem~\ref{thm:broadcast_matching}). Our policy achieves this guarantee with a static offer plan fixed before the first round. For this application, the constrained and unconstrained entries in Table~\ref{tab:application_improvements} refer to the same setting.

\paragraph{MNL prophet inequality.}
In the MNL prophet problem of \citet{goyal2023mnl}, products
have independent random revenue--weight pairs, revealed
upon arrival.
The firm must irrevocably accept or reject each product,
subject to a cardinality constraint on the final assortment.
The authors obtain a tight $1/2$ guarantee against a prophet
that observes all realizations before selecting its assortment.
We retain this factor against a stronger magician that may
also correlate the pairs across products while preserving
each product's distribution
(Theorem~\ref{thm:prophet_mnl}).

\subsection{Related work}\label{subsection:related_work}

\paragraph{Offline assortment optimization and the MNL choice model.}
In classical assortment optimization, the firm selects an assortment with known product parameters, and randomness arises from the customer's subsequent choice.
For the MNL choice model, \citet{talluri04} establish the optimality of revenue-ordered assortments in the unconstrained setting.
Subsequent work studies assortment optimization under various feasibility constraints \citep{rusmevichientong2010dynamic,davis2013assortment,desir2022capacitated}, as well as extensions involving random preferences \citep{rusmevichientong2014assortment}, assortment customization \citep{el2023joint,udwani2023submodular}, and product placement \citep{el2026location}.
Our framework provides a common way to analyze assortment decisions when the alternatives available for display are themselves random, with particular emphasis on the loss from committing before this uncertainty is resolved.

\paragraph{Random availability and assortment customization.}
Our committed firm's problem without cardinality constraints coincides with assortment optimization under the click-based MNL model of \citet{aouad2025click}.
In that model, each offered product independently enters the customer's consideration set, and the customer then chooses from this set according to MNL.
The authors establish NP-hardness and develop a polynomial-time approximation scheme.
Their guarantee concerns approximation of the optimal committed revenue.
Our focus is on comparing a simple committed policy with a stronger benchmark that can customize the assortment after observing availability and control the dependence among availability events.

\citet{ahmadnejadsaein2025adaptive} quantify the benefit of
controlling dependence through a tight factor-two
correlation-gap bound for the optimal MNL revenue function.
Their result corresponds to our delayed--magician comparison
in the unconstrained setting.

A closely related analytical approach appears in \citet{gallego2024assortment}, who develop linear programming bounds and independent-demand-based heuristics for assortment optimization under downward-closed constraints.
Their framework includes a random-attraction model that encompasses click-based MNL.
The key distinction concerns when feasibility is imposed.
Both their firm and their personalization benchmark select a feasible initial assortment before uncertainty is resolved; the benchmark may subsequently customize within that assortment.
In our setting, the cardinality constraint applies to the realized displayed assortment.
The committed firm may select a larger initial set, but any excess available products are removed adversarially, while the magician chooses its feasible assortment without an advance commitment to an initial set.
Through the distorted MNL choice model, we extend their unconstrained LP-based approach to this treatment of cardinality constraints while retaining the tight $1/2$ guarantee.
This robustness to how the realized assortment is truncated is central to our applications.

\paragraph{Prophet inequalities and online assortment optimization.}
Prophet inequalities provide a related perspective on the value of decisions made with incomplete information.
The classical results of \citet{krengel1978semiamarts} and \citet{samuel1984comparison} establish a tight $1/2$ guarantee, achievable by a static threshold policy.
This literature now encompasses many feasibility constraints, objectives, and arrival models; see \citet{correa2019recent} for a survey.

The closest online counterpart to our framework is the MNL prophet inequality of \citet{goyal2023mnl}.
Products arrive with independently drawn revenue--weight pairs, and the firm must accept or reject each product irrevocably after observing its parameters.
The authors obtain a tight $1/2$ guarantee against the optimal assortment chosen with full knowledge of all realizations, under a cardinality constraint and even against an almighty adversary.
Our policy uses the same distorted revenues as theirs.
The differences lie in the threshold used and the benchmark considered.
We compute our threshold directly from the input distributions, without evaluating the prophet's expected optimum, and retain the $1/2$ guarantee against a magician that may additionally choose the dependence across products while preserving their marginal parameter distributions.
Related work by \citet{wang2026prophet} develops prophet inequalities with stochastic choices, allowing broader choice rules and studying different constraints and forms of arrival information.

\vspace{2mm}
\underline{Organization:}
In Section~\ref{section:main_result}, we present our main algorithmic framework and prove Theorem~\ref{thm:mnl}.
In Sections~\ref{section:tsao}--\ref{section:mnl_prophet}, we apply this framework to the problems described above.
\section{Main result}\label{section:main_result}

In this section, we establish Theorem~\ref{thm:mnl} through a simple threshold policy based on distorted revenues.
In Section~\ref{subsection:distorted_mnl_choice_model}, we introduce the \emph{distorted MNL choice model}, an analytical tool for handling adversarial truncation.
In Section~\ref{subsection:lp_relaxation}, we derive an LP upper bound on the optimal value of the \ref{eq:opt_m} problem.
In Section~\ref{subsection:threshold_policy}, we present our threshold policy and prove Theorem~\ref{thm:mnl}.
Finally, in Section~\ref{subsection:ranked_mnl}, we extend our results to the \emph{ranked MNL choice model}, a generalization of MNL that we introduce to simplify the analysis of some of our applications.

\paragraph{Notation.}
We write $[n]=\{1,\ldots,n\}$ for the universe of products.
Throughout this section, we fix an availability probability vector
$\mathbf{x}\in[0,1]^n$, a preference weight vector
$\mathbf{v}\in\RR_+^n$, a revenue vector $\mathbf{r}\in\RR_+^n$,
and a cardinality limit $k\in\{1,2,\ldots\}\cup\{\infty\}$,
where $k=\infty$ denotes the unconstrained case.
For any $A\subseteq[n]$, we write $V(A)=\sum_{i\in A}v_i$.
For $a\in\RR$, we let $a^+=\max\{a,0\}$.
For an event or condition $E$, we write $\one\{E\}$ for its indicator.

\subsection{Distorted MNL choice model}
\label{subsection:distorted_mnl_choice_model}

We introduce an auxiliary choice model as an analytical tool
for handling adversarial truncation.
For each $S\subseteq[n]$, fix any minimizer $T(S)$ in the definition
of $R_k^-(S)$, so that $T(S)\subseteq S$,
$|T(S)|=\min\{|S|,k\}$, and $R(T(S))=R_k^-(S)$.
Thus, $T(S)=S$ when $|S|\le k$; otherwise, $T(S)$ is a
revenue-minimizing subset of $S$ of size $k$.

Using the convention $1/\infty=0$, we define the distorted
choice probability of each product $i\in[n]$ by
\[
    \widehat{\phi}_{\mathbf{v}}(i,S)
    =
    \one\{i\in T(S)\}
    \frac{v_i+1/k}{1+V(T(S))}.
\]
The corresponding exit probability is
\[
    \widehat{\phi}_{\mathbf{v}}(0,S)
    =
    \frac{(1-|S|/k)^+}{1+V(T(S))}.
\]
These probabilities are nonnegative and sum to one for every
assortment $S$, so $\widehat{\phi}_{\mathbf{v}}$ defines a valid
choice model, which we call the \emph{distorted MNL choice model}.
In general, this model does not have the MNL form.
When $k=\infty$, however, $T(S)=S$ and
$\widehat{\phi}_{\mathbf{v}}=\phi_{\mathbf{v}}$, so it coincides
with the original MNL choice model.

For each product $i\in[n]$, define its \emph{distorted weight}
and \emph{distorted revenue} by
\[
    \widehat v_i
    =
    \begin{cases}
        \displaystyle\frac{1+kv_i}{k-1},
        & 1<k<\infty,\\[2mm]
        +\infty,
        & k=1,\\
        v_i,
        & k=\infty,
    \end{cases}
    \qquad
    \widehat r_i
    =
    \begin{cases}
        \displaystyle\frac{kv_i}{1+kv_i}r_i,
        & k<\infty,\\[2mm]
        r_i,
        & k=\infty.
    \end{cases}
\]
We interpret $\widehat v_i/(1+\widehat v_i)$ as one when
$\widehat v_i=+\infty$.

These definitions yield two properties used in our analysis.
First, the distorted choice probabilities and revenues recover
the expected revenue under adversarial truncation:
\begin{equation}\label{eq:distorted_revenue}
    R_k^-(S)
    =
    \sum_{i\in S}
    \widehat{\phi}_{\mathbf{v}}(i,S)\widehat r_i
    \qquad\text{for every }S\subseteq[n].
\end{equation}
Indeed, for $k<\infty$,
\[
    \sum_{i\in S}
    \widehat{\phi}_{\mathbf{v}}(i,S)\widehat r_i
    =
    \sum_{i\in T(S)}
    \frac{v_i+1/k}{1+V(T(S))}
    \frac{kv_i}{1+kv_i}r_i
    =
    R(T(S))
    =
    R_k^-(S).
\]
For $k=\infty$, the identity follows because the choice
probabilities and revenues are unchanged.

Second, the distorted weights provide a uniform upper bound
on each product's choice probability:
\begin{equation}\label{eq:distorted_probability_bound}
    \max_{S\subseteq[n]:\,i\in S}
    \widehat{\phi}_{\mathbf{v}}(i,S)
    \le
    \frac{\widehat v_i}{1+\widehat v_i}
    \qquad\text{for every }i\in[n].
\end{equation}
To see this, suppose $1<k<\infty$.
If $i\notin T(S)$, its distorted choice probability is zero.
Otherwise, $V(T(S))\ge v_i$, so
\[
    \widehat{\phi}_{\mathbf{v}}(i,S)
    \le
    \frac{v_i+1/k}{1+v_i}
    =
    \frac{\widehat v_i}{1+\widehat v_i}.
\]
For $k=1$, the bound holds because the right-hand side equals one.
For $k=\infty$, it reduces to the standard MNL bound
$\phi_{\mathbf{v}}(i,S)\le v_i/(1+v_i)$.

Our threshold policy compares the distorted revenues
$\widehat r_i$ with a common threshold.
The same revenue distortion is used by \citet{goyal2023mnl}
in their analysis of MNL prophet inequalities under
cardinality constraints.

\subsection{LP relaxation}\label{subsection:lp_relaxation}

We construct an LP relaxation of the \ref{eq:opt_m} problem
using the distorted weights and revenues introduced in
Section~\ref{subsection:distorted_mnl_choice_model}.
To define this relaxation, consider the following primal--dual
pair for parameters $(\mathbf{x},\mathbf{v},\mathbf{r})$:
\begin{center}
    \begin{minipage}[t]{0.48\textwidth}
        \begin{align}
            \LP(\mathbf{x},\mathbf{v},\mathbf{r})=\max\quad
            &\sum_{i\in[n]}p_ir_i
            \label{eq:main_primal}\\
            \suchthat\quad
            &p_i\le x_i\frac{v_i}{1+v_i}
            &&\forall i\in[n],\notag\\
            &\sum_{i\in[n]}p_i\le1, \notag\\
            &\mathbf{p}\in\RR_+^n.\notag
        \end{align}
    \end{minipage}
    \hfill
    \begin{minipage}[t]{0.48\textwidth}
        \begin{align}
            \min\quad
            &\theta+
            \sum_{i\in[n]}
            x_i\frac{v_i}{1+v_i}\lambda_i
            \label{eq:main_dual}\\
            \suchthat\quad
            &\theta+\lambda_i\ge r_i
            &&\forall i\in[n],\notag\\
            &\theta\in\RR_+,
            \quad
            \boldsymbol{\lambda}\in\RR_+^n.\notag
        \end{align}
    \end{minipage}
\end{center}
We also consider the one-dimensional convex program
\begin{equation}\label{eq:main_convex}
    \min_{\theta\ge0}
    \left\{
        \theta+
        \sum_{i\in[n]}
        x_i\frac{v_i}{1+v_i}(r_i-\theta)^+
    \right\}.
\end{equation}

In the primal LP~\eqref{eq:main_primal}, the variable $p_i$
represents the unconditional probability that product $i$ is
chosen by the customer.
The individual constraint combines the upper bound $x_i$ on
its availability probability with its maximum MNL choice
probability $v_i/(1+v_i)$, attained when it is offered alone.
The constraint $\sum_{i\in[n]}p_i\le1$ reflects the fact that
the customer chooses at most one product.

\begin{remark}\label{remark:main_programs}
Programs~\eqref{eq:main_primal}, \eqref{eq:main_dual},
and~\eqref{eq:main_convex} have the same optimal value,
denoted by $\LP(\mathbf{x},\mathbf{v},\mathbf{r})$.
Moreover, Program~\eqref{eq:main_convex} admits an optimal
solution in $\{0,r_1,\ldots,r_n\}$.
\end{remark}

Indeed, the primal LP is feasible and bounded, so strong LP
duality gives equal optimal values for
Programs~\eqref{eq:main_primal} and~\eqref{eq:main_dual}.
For any fixed $\theta\ge0$, the smallest feasible value of
each dual variable $\lambda_i$ is $(r_i-\theta)^+$.
Since its objective coefficient is nonnegative, choosing
this value is optimal.
Substituting these values into Program~\eqref{eq:main_dual}
yields Program~\eqref{eq:main_convex}.
Its objective is continuous and piecewise linear, with
possible breakpoints at the revenues $r_i$, and equals
$\theta$ for $\theta\ge\max_{i\in[n]}r_i$.
It therefore attains its minimum at $0$ or at one of the
revenues.

The following lemma shows that evaluating this LP at the
distorted parameters gives the desired upper bound.

\begin{lemma}\label{lem:magician_upper_bound}
The magician's optimal expected revenue satisfies
\[
    \OPT_{\M}
    \le
    \LP(\mathbf{x},\widehat{\mathbf{v}},
        \widehat{\mathbf{r}}).
\]
\end{lemma}

\begin{proof}
We first establish the equivalent formulation
\begin{equation}\label{eq:magician_pruned}
    \OPT_{\M}
    =
    \max_{\alpha\in\cP(\mathbf{x})}
    \EE_{S\sim\alpha}[R_k^-(S)].
\end{equation}
Since $R_k^-(S)\le R_k^+(S)$ for every $S\subseteq[n]$,
the right-hand side is at most $\OPT_{\M}$ by the definition
of the \ref{eq:opt_m} problem.
For the reverse inequality, take an optimal distribution
$\alpha\in\cP(\mathbf{x})$ for that problem.
For each realized set $S$, choose a subset $U\subseteq S$
with $|U|\le k$ and $R(U)=R_k^+(S)$, and let $\beta$ denote
the distribution of $U$.
Because $U\subseteq S$, this selection cannot increase any
product's marginal availability probability, so
$\beta\in\cP(\mathbf{x})$.
Moreover, $|U|\le k$ implies $R_k^-(U)=R(U)$, and hence
\[
    \EE_{U\sim\beta}[R_k^-(U)]
    =
    \EE_{S\sim\alpha}[R_k^+(S)]
    =
    \OPT_{\M}.
\]
This proves~\eqref{eq:magician_pruned}.

Now take an optimal distribution
$\alpha\in\cP(\mathbf{x})$ in~\eqref{eq:magician_pruned}.
For each product $i\in[n]$, define its unconditional
choice probability under the distorted model by
\[
    p_i
    =
    \sum_{S\subseteq[n]}
    \alpha_S\widehat{\phi}_{\mathbf{v}}(i,S),
\]
where $\alpha_S$ is the probability assigned to $S$.
By~\eqref{eq:distorted_probability_bound} and
$\alpha\in\cP(\mathbf{x})$,
\[
    p_i
    \le
    \sum_{S\ni i}\alpha_S
    \frac{\widehat v_i}{1+\widehat v_i}
    \le
    x_i\frac{\widehat v_i}{1+\widehat v_i}.
\]
Furthermore, because the distorted choice probabilities,
including the exit probability, sum to one,
\[
    \sum_{i\in[n]}p_i
    =
    \sum_{S\subseteq[n]}\alpha_S
    \sum_{i\in[n]}\widehat{\phi}_{\mathbf{v}}(i,S)
    \le
    \sum_{S\subseteq[n]}\alpha_S
    =
    1.
\]
Thus, $\mathbf{p}$ is feasible for LP~\eqref{eq:main_primal}
with parameters
$(\mathbf{x},\widehat{\mathbf{v}},\widehat{\mathbf{r}})$.
By~\eqref{eq:distorted_revenue}, its objective value is
\begin{align*}
    \sum_{i\in[n]}p_i\widehat r_i
    &=
    \sum_{S\subseteq[n]}\alpha_S
    \sum_{i\in[n]}
    \widehat{\phi}_{\mathbf{v}}(i,S)\widehat r_i\\
    &=
    \sum_{S\subseteq[n]}\alpha_S R_k^-(S)
    =
    \OPT_{\M}.
\end{align*}
The optimal LP value is therefore at least $\OPT_{\M}$.
\end{proof}

\subsection{Threshold policy}\label{subsection:threshold_policy}

We now introduce the threshold policy that proves
Theorem~\ref{thm:mnl}.
Let $\tau\ge0$ satisfy
\begin{equation}\label{eq:tau}
    \tau
    =
    \sum_{i\in[n]}
    x_i\frac{\widehat v_i}{1+\widehat v_i}
    (\widehat r_i-\tau)^+,
\end{equation}
where $\widehat{\mathbf{v}}$ and $\widehat{\mathbf{r}}$
are the distorted weights and revenues.
The right-hand side is nonnegative, continuous, and
decreasing in $\tau$, and vanishes for
$\tau\ge\max_{i\in[n]}\widehat r_i$.
Since the left-hand side is strictly increasing,
Equation~\eqref{eq:tau} has a nonnegative solution.

Define the threshold set
\[
    H=\{i\in[n]:\widehat r_i\ge\tau\}.
\]
Our policy commits to $H$ before product availabilities
are realized.
This is a feasible solution to the \ref{eq:opt_c} problem:
the commitment set is unrestricted, and the realized set
of eligible available products $A\cap H$ is subject to
adversarial truncation.
We will show that this policy achieves at least half of
the magician's optimal expected revenue, thereby proving
Theorem~\ref{thm:mnl}.

To establish this guarantee, we prove a stronger statement
that also applies to negatively associated availability
distributions.
This class includes independent product availabilities
and allows us to use the same guarantee in subsequent
applications.
A distribution $\alpha$ on subsets of $[n]$ is
\emph{negatively associated} if, for any two increasing
functions $f,g:2^{[n]}\to\RR$ that depend on disjoint
coordinate sets,
\[
    \EE_{A\sim\alpha}[f(A)g(A)]
    \le
    \EE_{A\sim\alpha}[f(A)]\,
    \EE_{A\sim\alpha}[g(A)].
\]
See~\citet{pemantle2000towards} for background on negative
association and related notions of negative dependence.

The following lemma bounds the revenue of our threshold
policy under any negatively associated distribution whose
marginal availability probabilities are at least $x_i$
for products in $H$.

\begin{lemma}\label{lem:threshold_lower_bound}
Let $\alpha$ be a negatively associated distribution
with marginal probabilities $\mathbf{z}\in[0,1]^n$
satisfying $z_i\ge x_i$ for every $i\in H$.
Then
\[
    \frac12
    \LP(\mathbf{x},\widehat{\mathbf{v}},
        \widehat{\mathbf{r}})
    \le
    \tau
    \le
    \EE_{S\sim\alpha}\bigl[R_k^-(S\cap H)\bigr].
\]
\end{lemma}

\begin{proof}
For the first inequality, evaluate the objective of
Program~\eqref{eq:main_convex} at $\theta=\tau$.
By Remark~\ref{remark:main_programs} and
Equation~\eqref{eq:tau},
\[
    \LP(\mathbf{x},\widehat{\mathbf{v}},
        \widehat{\mathbf{r}})
    \le
    \tau+
    \sum_{i\in[n]}
    x_i\frac{\widehat v_i}{1+\widehat v_i}
    (\widehat r_i-\tau)^+
    =
    2\tau.
\]

We next prove the revenue lower bound.
Draw $S\sim\alpha$ and let $B=S\cap H$.
Since negative association is preserved under restriction
to a subset of coordinates, the distribution of $B$
is negatively associated.
Moreover, $\PP(i\in B)=z_i$ for every $i\in H$.
Define the unconditional choice probabilities under
the distorted model by
\[
    q_i=
    \EE\bigl[\widehat{\phi}_{\mathbf{v}}(i,B)\bigr]
    \quad(i\in H),
    \qquad
    q_0=
    \EE\bigl[\widehat{\phi}_{\mathbf{v}}(0,B)\bigr].
\]
Because the distorted model is a valid choice model,
these probabilities satisfy
$\sum_{i\in H}q_i=1-q_0$.

Fix $i\in H$.
If $i\in B$ and $|B\setminus\{i\}|<k$, then $|B|\le k$
and hence $T(B)=B$.
Therefore,
\begin{align*}
    q_i
    &=
    \left(v_i+\frac1k\right)
    \EE\left[
        \frac{\one\{i\in T(B)\}}{1+V(T(B))}
    \right] \ge
    \left(v_i+\frac1k\right)
    \EE\left[
        \frac{
            \one\{i\in B\}
            \one\{|B\setminus\{i\}|<k\}
        }{
            1+v_i+V(B\setminus\{i\})
        }
    \right].
\end{align*}
The indicator $\one\{i\in B\}$ is increasing in
coordinate $i$, whereas the remaining factor is
decreasing in the other coordinates.
Applying negative association to the first indicator
and the negative of the remaining factor gives
\begin{align*}
    q_i
    &\ge
    \left(v_i+\frac1k\right)
    \EE\bigl[\one\{i\in B\}\bigr]\,
    \EE\left[
        \frac{
            \one\{|B\setminus\{i\}|<k\}
        }{
            1+v_i+V(B\setminus\{i\})
        }
    \right]=
    z_i\frac{\widehat v_i}{1+\widehat v_i}
    \EE\left[
        \frac{
            (1+v_i)\one\{|B\setminus\{i\}|<k\}
        }{
            1+v_i+V(B\setminus\{i\})
        }
    \right].
\end{align*}

For every realization of $B$,
\[
    \widehat{\phi}_{\mathbf{v}}(0,B)
    =
    \frac{(1-|B|/k)^+}{1+V(T(B))}
    \le
    \frac{\one\{|B\setminus\{i\}|<k\}}{1+V(B)}
    \le
    \frac{
        (1+v_i)\one\{|B\setminus\{i\}|<k\}
    }{
        1+v_i+V(B\setminus\{i\})
    }.
\]
For the first inequality, the exit probability is zero
when $|B|\ge k$; when $|B|<k$, we have $T(B)=B$ and
the indicator equals one.
The second inequality follows by considering whether
$i\in B$.
Taking expectations and combining the preceding bounds
yields
\[
    q_i
    \ge
    z_i\frac{\widehat v_i}{1+\widehat v_i}q_0
    \ge
    x_i\frac{\widehat v_i}{1+\widehat v_i}q_0,
\]
where the last inequality uses $z_i\ge x_i$ for $i\in H$.

Using~\eqref{eq:distorted_revenue} and
$\widehat r_i\ge\tau$ for every $i\in H$, we obtain
\begin{align*}
    \EE[R_k^-(B)]
    =
    \sum_{i\in H}q_i\widehat r_i
    &=
    \tau\sum_{i\in H}q_i
    +
    \sum_{i\in H}q_i(\widehat r_i-\tau)\\
    &\ge
    \tau(1-q_0)
    +
    q_0\sum_{i\in H}
    x_i\frac{\widehat v_i}{1+\widehat v_i}
    (\widehat r_i-\tau)\\
    &=
    \tau(1-q_0)
    +
    q_0\sum_{i\in[n]}
    x_i\frac{\widehat v_i}{1+\widehat v_i}
    (\widehat r_i-\tau)^+=
    \tau,
\end{align*}
where the last line follows from~\eqref{eq:tau}.
This proves the revenue lower bound.
\end{proof}

We can now complete the proof of Theorem~\ref{thm:mnl}.

\begin{proof}[Proof of Theorem~\ref{thm:mnl}]
Product distributions are negatively associated because
functions of disjoint coordinates are independent.
Thus, under independent product availabilities,
$A\sim\cI(\mathbf{x})$ satisfies the assumptions of
Lemma~\ref{lem:threshold_lower_bound} with
$\mathbf{z}=\mathbf{x}$.
Combining this lemma with
Lemma~\ref{lem:magician_upper_bound} gives
\[
    \EE_{A\sim\cI(\mathbf{x})}[R_k^-(A\cap H)]
    \ge
    \tau
    \ge
    \frac12
    \LP(\mathbf{x},\widehat{\mathbf{v}},
        \widehat{\mathbf{r}})
    \ge
    \frac12\OPT_{\M}.
\]
Finally, $\tau$ can be computed in polynomial time by
sorting the distorted revenues, locating the interval
containing the solution of~\eqref{eq:tau}, and solving
the corresponding linear equation.
The set $H$ can then be constructed by comparing each
distorted revenue with $\tau$.
\end{proof}

\subsection{Extension to the ranked MNL choice model}
\label{subsection:ranked_mnl}

We extend our results to the \emph{ranked MNL choice model},
a generalization of MNL that simplifies the analysis of
several applications.
Each option $i\in[n]\cup\{0\}$ has a nonnegative integer
rank $s_i$ and a preference weight $v_i\ge0$.
We require $v_i>0$ whenever $s_i>0$.
The exit option has rank $s_0=0$ and weight $v_0=1$.
For an assortment $S\subseteq[n]$, let
\[
    S_{\mathbf{s}}
    =
    \left\{
        i\in S\cup\{0\}:
        s_i=\max_{h\in S\cup\{0\}}s_h
    \right\}
\]
denote the set of options with the highest rank.
The ranked MNL choice probabilities are
\[
    \phi_{\mathbf{v},\mathbf{s}}(i,S)
    =
    \one\{i\in S_{\mathbf{s}}\}
    \frac{v_i}{\sum_{h\in S_{\mathbf{s}}}v_h}
    \qquad(i\in[n]\cup\{0\}).
\]
Thus, the customer chooses among the highest-ranked
options, with probabilities proportional to their
preference weights.
When all products have rank zero, this model coincides
with the standard MNL choice model.
We define $R(S)$ and $R_k^\pm(S)$ using these choice
probabilities, with the same conventions as before.

The ranked MNL choice model can be obtained as a limit
of standard MNL models.
For $M>1$, assign each product $i\in[n]$ the MNL weight
$v_i(M)=v_iM^{s_i}$, keeping the exit weight equal to one.
For every assortment $S\subseteq[n]$ and option
$i\in[n]\cup\{0\}$,
\[
    \lim_{M\to\infty}
    \phi_{\mathbf{v}(M)}(i,S)
    =
    \phi_{\mathbf{v},\mathbf{s}}(i,S).
\]
The corresponding limits of the distorted weights
and revenues are
\[
    \widehat v_{i,\mathbf{s}}
    =
    \begin{cases}
        \displaystyle\frac{1+kv_i}{k-1},
        & s_i=0 \text{ and }1<k<\infty,\\[2mm]
        +\infty,
        & s_i>0 \text{ or }k=1,\\
        v_i,
        & s_i=0 \text{ and }k=\infty,
    \end{cases}
    \qquad
    \widehat r_{i,\mathbf{s}}
    =
    \begin{cases}
        \displaystyle\frac{kv_i}{1+kv_i}r_i,
        & s_i=0 \text{ and }k<\infty,\\[2mm]
        r_i,
        & s_i>0 \text{ or }k=\infty.
    \end{cases}
\]
As before, we interpret
$\widehat v_{i,\mathbf{s}}/(1+\widehat v_{i,\mathbf{s}})$
as one when $\widehat v_{i,\mathbf{s}}=+\infty$.
The threshold $\tau$ is defined by
Equation~\eqref{eq:tau} with these distorted parameters,
and the policy commits to
$H=\{i\in[n]:\widehat r_{i,\mathbf{s}}\ge\tau\}$.

Convergence of the choice probabilities implies convergence of the
assortment revenues and the optimal benchmark values.
The LP values also converge through the coefficients
$\widehat v_i/(1+\widehat v_i)$ and the distorted revenues.
Consequently, Theorem~\ref{thm:mnl},
Lemma~\ref{lem:magician_upper_bound}, and
Lemma~\ref{lem:threshold_lower_bound} extend to the
ranked MNL choice model.

When all products have rank zero, or when the rank vector
is clear from context, we omit $\mathbf{s}$ from the
subscripts.

\section{Two-sided assortment optimization}\label{section:tsao}

Two-sided assortment optimization (TSAO), introduced by
\citet{ashlagi_etal19}, concerns matching between two sets of agents:
a set $\cC$ of customers and a set $\cS$ of suppliers. The platform offers each agent
an assortment of agents from the other side of the market. Each
agent then selects an offered agent or the exit option according
 to its own choice model. 
 A match forms when two agents select each other. The
platform seeks to maximize the expected total revenue from these
matches; maximizing the expected number of matches is the special
case in which every match has unit revenue.
We apply the framework of Section~\ref{section:main_result} to two variants of TSAO: standard TSAO \citep{ashlagi_etal19,torrico21} and fully static TSAO \citep{ahmed2022parameterized,housni2024two}.
For both variants, we consider a more general setting that allows pair-dependent revenues together with cardinality constraints on both sides of the market.
We obtain a $1/2$ guarantee for the standard model and a $1/4$ guarantee for the fully static model.

\subsection{Standard two-sided assortment optimization}\label{section:standard_tsao}

We use \emph{standard TSAO} to describe models in which one side of the market initiates the matching process and the other side responds.
Throughout this subsection, customers are the initiating side: the platform offers them assortments of suppliers, and each customer selects a supplier or the exit option.
After all customers have chosen, each supplier chooses among the customers who selected it, subject to the platform's rule for displaying these customers.

\citet{ashlagi_etal19} introduced this setting with static customer assortments, fixed before any choices are observed, and obtained the first constant-factor approximation for maximizing the expected number of matches.
\citet{torrico21} improved the approximation guarantee and extended the setting to more general preferences, supplier-dependent revenues, and constraints on customer assortments.
\citet{housni2024two} studied adaptive policies, in which customers are processed one at a time and subsequent assortments may depend on earlier choices.
They obtained constant-factor guarantees for maximizing the expected number of matches and studied the value of adaptivity.
More recent work allows arbitrary pair-dependent revenues under MNL choice.
\citet{nissim2024revenue} study static customer assortments under two supplier display rules: the \emph{inclusive} rule displays every customer who selected a supplier, whereas the \emph{customized} rule allows the platform to select a subset.
In a recent revision of the same paper, concurrent with our work, they obtain a $1/2$ guarantee for an inclusive policy against the optimal static customized policy.
\citet{ahmadnejadsaein2025adaptive} study adaptive policies with customization and obtain a static $(1/2-\varepsilon)$-approximation against the optimal adaptive customized policy.
These results for arbitrary pair-dependent revenues do not impose cardinality constraints.

Our framework gives a common analysis of these MNL models and yields a simple static policy with a $1/2$ guarantee in a setting that combines pair-dependent revenues with cardinality constraints on both sides.
The result improves approximation factors or strengthens benchmarks relative to prior revenue-maximization results, as summarized in Table~\ref{tab:application_improvements}, while extending the guarantees to constrained assortments.
Cardinality constraints introduce an additional difficulty: feasible customer assortments can generate supplier backlogs that exceed the supplier display limits.
Removing customers from these backlogs changes the supplier's choice probabilities and may reduce expected revenue.
Our guarantee accounts for this loss even when the retained customers are selected adversarially.

\paragraph{Constrained standard TSAO.}
There is a set $\cC$ of customers and a set $\cS$ of suppliers.
Each customer $i\in\cC$ has an MNL choice model with preference weights $(v_{ij})_{j\in\cS}$, and each supplier $j\in\cS$ has an MNL choice model with preference weights $(w_{ji})_{i\in\cC}$.
All weights are nonnegative, and each agent's exit option has weight one.
A match between customer $i$ and supplier $j$ generates revenue $r_{ij}\ge0$.
The platform seeks to maximize expected total revenue.
Each agent's choice is sampled from its MNL model using a fresh random draw independent of the draws used for other agents.

The platform offers each customer $i$ an assortment $S_i\subseteq\cS$ with $|S_i|\le k$, where $k$ is a positive integer.
A customer policy is \emph{static} if all customer assortments are fixed before any customer chooses.
It is \emph{adaptive} if customers are processed one at a time and the platform may choose the next customer and its assortment based on the choices observed so far.
Customer $i$ selects supplier $j\in S_i$ with probability
$
    v_{ij}/(1+\sum_{h\in S_i}v_{ih}),
$
and selects the exit option with the remaining probability.

After all customers have chosen, supplier $j$ has the random \emph{backlog}
$
    B_j=\{i\in\cC:i\text{ chose }j\}.
$
The assortment displayed to supplier $j$ must be a subset $T_j\subseteq B_j$ with $|T_j|\le k$.
Supplier $j$ selects customer $i\in T_j$ with probability
$
    w_{ji}/(1+\sum_{h\in T_j}w_{jh}),
$
and selects the exit option with the remaining probability.
Selecting customer $i$ forms a match because that customer has already selected supplier~$j$.

We distinguish two rules for determining $T_j$.
Under \emph{customization}, the platform may choose any subset of $B_j$ containing at most $k$ customers.
Under \emph{adversarial truncation}, the entire backlog is displayed when $|B_j|\le k$; otherwise, an adversary selects the $k$ customers whose display minimizes expected revenue from supplier $j$.
When $k\ge|\cC|$, adversarial truncation displays every supplier's entire backlog and coincides with the inclusive rule.
The unconstrained model is recovered by taking $k\ge\max\{|\cC|,|\cS|\}$.
We use a common limit $k$ for simplicity; the result also extends to agent-dependent cardinality constraints.

Let $\OPT_{\AC}$ denote the maximum expected revenue of an adaptive customer policy with customized supplier assortments, subject to these cardinality constraints.
We compare our static policy, evaluated under adversarial truncation, with this adaptive customized benchmark.

\begin{theorem}\label{thm:standard_tsao}
For constrained standard TSAO, there is a randomized polynomial-time static policy whose expected revenue under adversarial truncation is at least $\frac12\OPT_{\AC}$.
\end{theorem}

Thus, even the optimal policy with adaptive customer assortments and customized supplier assortments earns at most twice the expected revenue of our static policy under adversarial truncation.
To the best of our knowledge, this is the first constant-factor approximation for standard TSAO with arbitrary pair-dependent revenues and cardinality constraints on both sides of the market.

\paragraph{Our solution.}
We apply the framework of Section~\ref{section:main_result} separately to each supplier.
For a fixed supplier $j$, view each customer $i$ as a product with preference weight $w_{ji}$ and revenue $r_{ij}$, available when customer $i$ chooses $j$.
Under an adaptive customer policy, these availability events may be dependent.
The \textsf{\ref{eq:opt_m}} benchmark allows such dependence, so Lemma~\ref{lem:magician_upper_bound} provides an upper bound on each supplier's contribution.
For our static policy, we sample customer assortments independently and implement the threshold commitments of Lemma~\ref{lem:threshold_lower_bound} by filtering these assortments before displaying them. 

The algorithm therefore has two tasks: select feasible customer choice probabilities through an LP, and construct customer assortments that satisfy the hypotheses of the threshold guarantee.
The characterization below connects the LP variables to feasible random assortments.

\paragraph{Feasible customer choice probabilities.}
Let $x_{ij}$ represent the probability that customer $i$ chooses supplier $j$, equivalently the availability probability of customer $i$ in the single-customer instance associated with supplier $j$.
For an array $\mathbf{x}=(x_{ij})_{i\in\cC,j\in\cS}$, write $\mathbf{x}_i=(x_{ij})_{j\in\cS}$ for a customer row and $\mathbf{x}_j=(x_{ij})_{i\in\cC}$ for a supplier column.
We similarly write $\mathbf{v}_i=(v_{ij})_{j\in\cS}$, $\mathbf{w}_j=(w_{ji})_{i\in\cC}$, and $\mathbf{r}_j=(r_{ij})_{i\in\cC}$.
For $A\subseteq\cS$ and $B\subseteq\cC$, let $V_i(A)=\sum_{j\in A}v_{ij}$ and $W_j(B)=\sum_{i\in B}w_{ji}$.
The expected revenue from displaying $B$ to supplier $j$ is
\[
    R_j(B)
    =
    \frac{\sum_{i\in B}r_{ij}w_{ji}}{1+W_j(B)}.
\]
We use $R_{j,k}^-$ and $R_{j,k}^+$ for the corresponding adversarial-truncation and customization revenue functions.

A supplier with $v_{ij}=0$ is never chosen by customer $i$ and may be removed from its assortment without changing its choice probabilities.
Throughout, we fix $x_{ij}=0$ whenever $v_{ij}=0$ and interpret the corresponding term $x_{ij}/v_{ij}$ as zero.
For each customer $i\in\cC$, define
\[
    \cA_i
    =
    \left\{
        \mathbf{x}_i\in\RR_+^{\cS}:
            \frac{x_{ij}}{v_{ij}}
            +\sum_{h\in\cS}x_{ih}
            \le1 \; 
            \forall j\in\cS,
            \sum_{j\in\cS}\frac{x_{ij}}{v_{ij}}
            +k\sum_{h\in\cS}x_{ih}
            \le k    
    \right\}.
\]

This polytope is the cardinality-constrained specialization of the MNL formulation of \citet{davis2013assortment}, after eliminating the exit-probability variable.
The following lemma gives both its interpretation and the decomposition needed by our algorithm.

\begin{lemma}\label{lem:standard_choice_polytope}
The polytope $\cA_i$ is the convex hull of the MNL choice-probability vectors of assortments $A\subseteq\cS$ with $|A|\le k$.
Moreover, for every $\mathbf{x}_i\in\cA_i$, a distribution for a random set $A_i$ supported on polynomially many such assortments and satisfying
$
    \EE\bigl[\phi_{\mathbf{v}_i}(j,A_i)\bigr]
    =
    x_{ij}   
$ for $ j\in\cS $
can be computed in polynomial time.
\end{lemma}

The proof of Lemma~\ref{lem:standard_choice_polytope} is included in Appendix~\ref{appendix:standard_choice_polytope}.
The convex-hull characterization ensures that the unconditional choice vector of any adaptive customer policy belongs to $\cA_i$.
Conversely, the decomposition allows us to implement any vector in $\cA_i$ by sampling a feasible customer assortment.

\paragraph{LP relaxation.}
For each supplier $j$, obtain the distorted weights $\widehat w_{ji}$ and revenues $\widehat r_{ij}$ by applying the definitions in Section~\ref{subsection:distorted_mnl_choice_model} to weights $w_{ji}$, revenues $r_{ij}$, and cardinality limit $k$.
Write $\widehat{\mathbf{w}}_j$ and $\widehat{\mathbf{r}}_j$ for the corresponding vectors, and let $\Theta_j=\{0\}\cup\{\widehat r_{ij}:i\in\cC\}$.

For given customer choice probabilities $\mathbf{x}$, Lemma~\ref{lem:magician_upper_bound} bounds supplier $j$'s contribution by
$\LP(\mathbf{x}_j,\widehat{\mathbf{w}}_j,\widehat{\mathbf{r}}_j)$.
Combining these supplier bounds with the customer choice polytopes gives the relaxation
\begin{align}
    \LP_{\ST}=\max\quad
    &\sum_{j\in\cS}\zeta_j
    \label{eq:lp_standard}\\
    \suchthat\quad
    &\zeta_j\le
      \theta+
      \sum_{i\in\cC}
      x_{ij}\frac{\widehat w_{ji}}{1+\widehat w_{ji}}
      (\widehat r_{ij}-\theta)^+
      &&\forall j\in\cS,\ \forall\theta\in\Theta_j,
      \notag\\
    &\mathbf{x}_i\in\cA_i
      &&\forall i\in\cC,
      \notag\\
    &\mathbf{x}\in\RR_+^{\cC\times\cS},
      \qquad
      \boldsymbol{\zeta}\in\RR_+^{\cS}.
      \notag
\end{align}
By Remark~\ref{remark:main_programs}, the constraints indexed by $\Theta_j$ are equivalent to
\[
    \zeta_j
    \le
    \LP\bigl(
        \mathbf{x}_j,
        \widehat{\mathbf{w}}_j,
        \widehat{\mathbf{r}}_j
    \bigr).
\]
For fixed $\mathbf{x}$, maximizing over $\boldsymbol{\zeta}$ makes each of these inequalities tight.
Thus, LP~\eqref{eq:lp_standard} maximizes the sum of the supplier upper bounds over feasible customer choice probabilities.
The LP has polynomially many variables and constraints.

\paragraph{Threshold policy.}
Take an optimal solution $(\mathbf{x},\boldsymbol{\zeta})$ of LP~\eqref{eq:lp_standard}.
For each supplier $j$, apply the threshold construction of Section~\ref{subsection:threshold_policy} with reference availability probabilities $\mathbf{x}_j$, weights $\mathbf{w}_j$, and revenues $\mathbf{r}_j$.
This gives a threshold $\tau_j$ and a set $H_j$ of eligible customers.

We implement these supplier commitments through the customer assortments.
For each customer $i$, independently sample a feasible assortment $A_i$ using Lemma~\ref{lem:standard_choice_polytope}, and retain supplier $j$ only if $i\in H_j$.
The resulting assortment is $S_i=\{j\in A_i:i\in H_j\}$.
This filtering preserves the customer cardinality constraint and ensures that supplier $j$'s backlog is contained in $H_j$.
Moreover, removing other suppliers from a customer's assortment can only increase the choice probability of a retained supplier.
The actual availability probabilities therefore dominate the reference probabilities on each threshold set, as required by Lemma~\ref{lem:threshold_lower_bound}. The steps of our policy are summarized in Algorithm \ref{alg:standard_tsao}.

\begin{algorithm}[!htbp]
    \caption{Static threshold policy for standard TSAO}
    \label{alg:standard_tsao}
    \begin{algorithmic}[1]
        \State Solve LP~\eqref{eq:lp_standard} and obtain
        an optimal solution $(\mathbf{x},\boldsymbol{\zeta})$.
        \For{each supplier $j\in\cS$}
            \State Compute the threshold $\tau_j\ge0$ satisfying
            \Statex \hspace{\algorithmicindent}
            $\displaystyle
                \tau_j=
                \sum_{i\in\cC}
                x_{ij}
                \frac{\widehat w_{ji}}{1+\widehat w_{ji}}
                (\widehat r_{ij}-\tau_j)^+.
            $
            \State Set
            $H_j=\{i\in\cC:\widehat r_{ij}\ge\tau_j\}$.
        \EndFor
        \For{each customer $i\in\cC$}
            \State Compute the assortment distribution for
            $\mathbf{x}_i$ from
            Lemma~\ref{lem:standard_choice_polytope}.
            \State Independently sample an assortment $A_i$
            from this distribution.
            \State Set $S_i=\{j\in A_i:i\in H_j\}$.
        \EndFor
        \State Display $S_i$ to every customer $i\in\cC$
        and observe the customer choices.
        \For{each supplier $j\in\cS$}
            \State Display its backlog $B_j$, subject to
            adversarial truncation to $\min\{|B_j|,k\}$ customers.
        \EndFor
    \end{algorithmic}
\end{algorithm}

\begin{proof}[Proof of Theorem~\ref{thm:standard_tsao}]
Let $\ALG_{\ST}$ denote the expected revenue of Algorithm~\ref{alg:standard_tsao}.
We prove
\[
    \OPT_{\AC}
    \le
    \LP_{\ST}
    \le
    2\ALG_{\ST}.
\]

\emph{Upper bound on the adaptive customized optimum.}
Fix an adaptive customized policy, and let $x_{ij}$ be the probability that customer $i$ chooses supplier $j$.
Conditional on the history before customer $i$ is processed and on the assortment offered to it, its choice probabilities are those of a feasible MNL assortment.
Averaging over these histories and assortments expresses $\mathbf{x}_i$ as a convex combination of feasible choice-probability vectors.
Hence $\mathbf{x}_i\in\cA_i$ by Lemma~\ref{lem:standard_choice_polytope}.

Fix a supplier $j$.
Its backlog has marginal inclusion probabilities $\mathbf{x}_j$, although its inclusion indicators may be dependent.
Its backlog distribution therefore belongs to $\cP(\mathbf{x}_j)$.
Thus, the policy's expected contribution from supplier $j$ is at most
\begin{align*}
    \max_{\alpha\in\cP(\mathbf{x}_j)}
    \EE_{B\sim\alpha}[R_{j,k}^+(B)]
    &\le
    \LP\bigl(
        \mathbf{x}_j,
        \widehat{\mathbf{w}}_j,
        \widehat{\mathbf{r}}_j
    \bigr)=
    \min_{\theta\in\Theta_j}
    \left\{
        \theta+
        \sum_{i\in\cC}
        x_{ij}\frac{\widehat w_{ji}}{1+\widehat w_{ji}}
        (\widehat r_{ij}-\theta)^+
    \right\}.
\end{align*}
The inequality follows from Lemma~\ref{lem:magician_upper_bound}, and the equality follows from Remark~\ref{remark:main_programs}.
Setting $\zeta_j$ equal to the policy's expected contribution from supplier $j$ therefore gives a feasible solution of LP~\eqref{eq:lp_standard}.
Its objective equals the policy's expected total revenue.
Since the policy was arbitrary, this proves $\OPT_{\AC}\le\LP_{\ST}$.

\emph{Performance of the static policy.}
Take an optimal solution $(\mathbf{x},\boldsymbol{\zeta})$ of LP~\eqref{eq:lp_standard} and run Algorithm~\ref{alg:standard_tsao}.
All customer assortments are fixed before any customer chooses.
Moreover, $S_i\subseteq A_i$ and $|A_i|\le k$ imply $|S_i|\le k$.
Thus, the algorithm defines a feasible static policy.

Let $z_{ij}$ be the resulting probability that customer $i$ chooses supplier $j$.
Before filtering, the sampled assortment $A_i$ satisfies
$\EE[\phi_{\mathbf{v}_i}(j,A_i)]=x_{ij}$ by Lemma~\ref{lem:standard_choice_polytope}.
For $i\in H_j$, filtering retains $j$ whenever $j\in A_i$ and removes only other suppliers.
Consequently,
\[
    z_{ij}
    =
    \EE[\phi_{\mathbf{v}_i}(j,S_i)]
    \ge
    \EE[\phi_{\mathbf{v}_i}(j,A_i)]
    =
    x_{ij}
    \qquad(i\in H_j).
\]
Also, a customer can choose supplier $j$ only if it belongs to $H_j$, so $B_j\subseteq H_j$.

For a fixed supplier $j$, each filtered assortment $S_i$ depends only on the independently sampled assortment $A_i$ and the fixed threshold sets.
The customer choice draws are also independent across customers.
Thus, the inclusion indicators of $B_j$ are independent, and hence negatively associated, with marginal probabilities $\mathbf{z}_j=(z_{ij})_{i\in\cC}$.
Apply Lemma~\ref{lem:threshold_lower_bound} with reference probabilities $\mathbf{x}_j$, actual probabilities $\mathbf{z}_j$, weights $\mathbf{w}_j$, revenues $\mathbf{r}_j$, and threshold set $H_j$.
We obtain
\[
    \zeta_j
    \le
    \LP\bigl(
        \mathbf{x}_j,
        \widehat{\mathbf{w}}_j,
        \widehat{\mathbf{r}}_j
    \bigr)
    \le
    2\tau_j
    \le
    2\EE[R_{j,k}^-(B_j\cap H_j)]
    =
    2\EE[R_{j,k}^-(B_j)].
\]
The first inequality follows from the supplier constraints in LP~\eqref{eq:lp_standard} and Remark~\ref{remark:main_programs}.
The second and third inequalities follow from Lemma~\ref{lem:threshold_lower_bound}.
The equality uses $B_j\subseteq H_j$.
Summing over suppliers gives
\[
    \LP_{\ST}
    =
    \sum_{j\in\cS}\zeta_j
    \le
    2\sum_{j\in\cS}\EE[R_{j,k}^-(B_j)]
    =
    2\ALG_{\ST}.
\]

Finally, LP~\eqref{eq:lp_standard} has polynomial size, Lemma~\ref{lem:standard_choice_polytope} provides a polynomial-time assortment decomposition, and each threshold can be computed by sorting the distorted revenues.
The remaining sampling and filtering steps also take polynomial time.
Hence Algorithm~\ref{alg:standard_tsao} runs in polynomial time.

\end{proof}

\subsection{Fully static two-sided assortment optimization}
\label{subsection:fully_static}

In fully static TSAO, the platform fixes assortments for all
customers and suppliers before any agent makes a choice. Each
agent then selects an offered agent or the exit option, and a
match forms when two agents select each other.
This model was studied by \citet{ahmed2022parameterized}, who
obtain parameterized approximation guarantees, and
\citet{housni2024two}, who give a $0.067$-approximation for
uniform match revenues under cardinality constraints on both
sides of the market. Applying the framework of
Section~\ref{section:main_result}, we improve the latter
guarantee to $1/4$ while allowing arbitrary pair-dependent
revenues.

\paragraph{Constrained fully static TSAO.}
There is a set $\cC$ of customers and a set $\cS$ of suppliers.
Each customer $i\in\cC$ follows an MNL choice model with
preference weights $(v_{ij})_{j\in\cS}$, and each supplier
$j\in\cS$ follows an MNL choice model with preference weights
$(w_{ji})_{i\in\cC}$. All preference weights are nonnegative,
and each agent's exit option has weight one. A match between
customer $i$ and supplier $j$ generates revenue $r_{ij}\ge0$.
The platform offers each customer $i$ an assortment
$S_i\subseteq\cS$ and each supplier $j$ an assortment
$C_j\subseteq\cC$, subject to
$
    |S_i|\le k \quad \forall i\in\cC
$ and 
$   |C_j|\le k \quad \forall j\in\cS.$
The platform may randomize over these assortments, but all
assortments must be fixed before any agent makes a choice.
Conditional on the chosen assortments, agents choose
independently according to their respective MNL models.
Customer $i$ and supplier $j$ match if and only if they select
each other.
The objective is to maximize the expected total revenue from
matches. Let $\OPT_{\FS}$ denote the maximum expected revenue
over all fully static policies satisfying the cardinality
constraints.

\begin{theorem}\label{thm:fully_static}
For constrained fully static TSAO with pair-dependent
revenues, there is a randomized polynomial-time fully static
policy whose expected revenue is at least
$\frac14\OPT_{\FS}$.
\end{theorem}

\paragraph{Our solution.}
Our approach applies the threshold guarantee from Section~\ref{section:main_result} on both sides of the market.
We first sample customer assortments independently and use the guarantee to bound their expected contribution, weighted by each supplier's probability of choosing a customer offered alone.
We then choose the best supplier assortments for the sampled customer assortments.
A second application of the threshold guarantee bounds the revenue retained in this step.  The supplier-side threshold
policy is used only in the analysis: for each realization of
the sampled customer assortments, our algorithm chooses
optimal supplier assortments and therefore performs at least
as well as this policy.

Let $\widehat v_{ij}$ and $\widehat w_{ji}$ be the distorted preference weights obtained from $v_{ij}$ and $w_{ji}$, respectively, with cardinality limit $k$, as defined in Section~\ref{subsection:distorted_mnl_choice_model}.
For every pair $(i,j)\in\cC\times\cS$, write
\[
    a_{ij}=\frac{v_{ij}}{1+v_{ij}},
    \qquad
    b_{ij}=\frac{w_{ji}}{1+w_{ji}},
    \qquad
    \widehat a_{ij}=\frac{\widehat v_{ij}}{1+\widehat v_{ij}},
    \qquad
    \widehat b_{ij}=\frac{\widehat w_{ji}}{1+\widehat w_{ji}}.
\]
Here, $a_{ij}$ and $b_{ij}$ are the choice probabilities when the other agent is offered alone; $\widehat a_{ij}$ and $\widehat b_{ij}$ are the corresponding distorted probabilities.
The distortion formulas give
\[
    \begin{aligned}
        \widehat a_{ij}=\frac{v_{ij}+1/k}{1+v_{ij}},
        \quad\widehat b_{ij}=\frac{w_{ji}+1/k}{1+w_{ji}},\quad\widehat a_{ij}\frac{kv_{ij}}{1+kv_{ij}}=a_{ij},\quad\widehat b_{ij}\frac{kw_{ji}}{1+kw_{ji}}&=b_{ij}.
    \end{aligned}
\]
We use the following relaxation:
\begin{align}
    \LP_{\FS}=\max\quad
    &\sum_{i\in\cC}\sum_{j\in\cS}r_{ij}a_{ij}b_{ij}x_{ij}
    \label{eq:lp_fully_static}\\
    \suchthat\quad
    &\sum_{j\in\cS}\widehat a_{ij}x_{ij}\le1
    &&\forall i\in\cC,
    \notag\\
    &\sum_{i\in\cC}\widehat b_{ij}x_{ij}\le1
    &&\forall j\in\cS,
    \notag\\
    &\mathbf{x}\in[0,1]^{\cC\times\cS}.
    \notag
\end{align}
The customer and supplier constraints make the total-choice constraint in LP~\eqref{eq:main_primal} redundant in the two threshold applications below.
LP~\eqref{eq:lp_fully_static} has polynomially many variables and constraints.

For each customer, we sample suppliers independently,
retain those passing the threshold, and truncate the
resulting assortment to size at most $k$.
We perform this sampling independently across customers and use a fixed truncation rule for each customer based only on its own sampled set.
Once these customer assortments are fixed, the expected contribution of a choice of customer $i$ by supplier $j$ is $r_{ij}\phi_{\mathbf{v}_i}(j,S_i)$.
Thus, selecting a supplier assortment is a cardinality-constrained MNL revenue maximization problem.

Algorithm \ref{alg:fully_static_tsao} summarizes the steps of our solution. The details of the proof of Theorem~\ref{thm:fully_static} are deferred to Appendix~\ref{appendix:fully_static_proof}.

\begin{algorithm}[H]
    \caption{Fully static threshold policy}
    \label{alg:fully_static_tsao}
    \begin{algorithmic}[1]
        \State Solve LP~\eqref{eq:lp_fully_static} and obtain an optimal solution $\mathbf{x}$.
        \For{each customer $i\in\cC$}
            \State Compute the threshold $\tau_i\ge0$ satisfying
            \Statex \hspace{\algorithmicindent}$\displaystyle
                \tau_i=\sum_{j\in\cS}x_{ij}\widehat a_{ij}
                \left(\frac{kv_{ij}}{1+kv_{ij}}r_{ij}b_{ij}-\tau_i\right)^+.
            $
            \State Set $\displaystyle H_i=\left\{j\in\cS:\frac{kv_{ij}}{1+kv_{ij}}r_{ij}b_{ij}\ge\tau_i\right\}$.
            \State Independently sample $A_i\sim\cI(\mathbf{x}_i)$.
            \State Let $S_i$ be the first $\min\{|A_i\cap H_i|,k\}$ suppliers in $A_i\cap H_i$ under a fixed order.
        \EndFor
        \For{each supplier $j\in\cS$}
            \State Choose
            \Statex \hspace{\algorithmicindent}$\displaystyle
                C_j\in\argmax_{\substack{C\subseteq\{i\in\cC:j\in S_i\}\\|C|\le k}}
                \sum_{i\in C}r_{ij}
                \frac{v_{ij}}{1+V_i(S_i)}
                \frac{w_{ji}}{1+W_j(C)}.
            $
        \EndFor
        \State Display $S_i$ to every customer $i\in\cC$ and $C_j$ to every supplier $j\in\cS$.
    \end{algorithmic}
\end{algorithm}

\section{MMNL with customization}
\label{section:mmnl}

In assortment optimization under a mixture of multinomial
logit models (MMNL) with customization, the firm selects an
initial assortment before observing the customer's class.
After observing the class, it may remove products from this
assortment to tailor the displayed selection to the customer.
This model was introduced by \citet{el2023joint}.

\paragraph{Constrained MMNL with customization.}
There are $n$ products and $m$ customer classes.
A customer belongs to class $j\in[m]$ with probability
$\nu_j\ge0$, where $\sum_{j\in[m]}\nu_j=1$.
Class $j$ follows an MNL choice model with preference weights
$\mathbf{v}_j=(v_{ji})_{i\in[n]}\in\RR_+^n$ and exit
weight one.
If a class-$j$ customer chooses product $i$, the firm earns
revenue $r_{ij}\ge0$.

Before observing the customer's class, the firm selects an
initial assortment $S\subseteq[n]$ containing at most
$\ell$ products, where $\ell\in\{1,\ldots,n\}$.
After observing the class, it chooses an assortment
$T\subseteq S$ containing at most $k$ products to display,
where $k$ is a positive integer.
Thus, $\ell$ limits the initial assortment, while $k$
limits the assortment displayed to the customer.
The case without a display limit is recovered by taking
$k\ge\ell$.

Write $\mathbf{r}_j=(r_{ij})_{i\in[n]}$ and
$R_j=R_{\mathbf{r}_j,\mathbf{v}_j}$, and let
$R_{j,k}^-$ and $R_{j,k}^+$ denote the corresponding
adversarial-truncation and customization revenue functions
defined in Section~\ref{section:introduction}.
The firm seeks to maximize its expected revenue:
\begin{equation}\label{eq:mmnl}
    \begin{aligned}
        \OPT_{\MMNL}
        &=
        \max_{S\subseteq[n],\;|S|\le\ell}
        \sum_{j\in[m]}\nu_j
        \max_{T\subseteq S,\;|T|\le k}
        R_{\mathbf{r}_j,\mathbf{v}_j}(T)\\
        &=
        \max_{S\subseteq[n],\;|S|\le\ell}
        \sum_{j\in[m]}\nu_j R_{j,k}^+(S).
    \end{aligned}
\end{equation}

When revenues are independent of the customer's class and
there is no display limit, \citet{el2023joint} give an
$\Omega(1/\log m)$-approximation algorithm.
For this setting, \citet{udwani2023submodular} improves the
guarantee to $(1/2-\varepsilon)$ using submodular order functions.
With a display limit and class-independent revenues,
\citet{el2023joint} obtain an
$\Omega(1/\sqrt{m\log m})$-approximation guarantee.

Applying the framework of Section~\ref{section:main_result},
we obtain a $1/2$-approximation for
Problem~\eqref{eq:mmnl}.
This guarantee accommodates cardinality constraints on both
the initial and displayed assortments and allows revenues
to depend on both the product and the customer's class.

\begin{theorem}\label{thm:mmnl}
There is a randomized polynomial-time $1/2$-approximation
algorithm for Problem~\eqref{eq:mmnl}.
\end{theorem}

The result also extends to class-dependent display limits
$k_j$; we use a common limit $k$ to simplify the exposition.

\paragraph{Our solution.}
For each class $j$, view membership in the initial assortment
as product availability and apply the framework of
Section~\ref{section:main_result} with preference weights
$\mathbf{v}_j$, revenues $\mathbf{r}_j$, and display limit $k$.
The classes share the same initial assortment, so we use a
common vector of marginal inclusion probabilities.
We optimize these probabilities through an LP and then sample
an initial assortment of size $\ell$ whose inclusion
indicators are negatively associated.
This allows us to apply Lemma~\ref{lem:threshold_lower_bound}
separately to each class.

For each class $j$, obtain the distorted weights
$\widehat{\mathbf{v}}_j$ and revenues
$\widehat{\mathbf{r}}_j$ from the definitions in
Section~\ref{subsection:distorted_mnl_choice_model}, using
display limit $k$, and set
$\Theta_j=\{0\}\cup\{\widehat r_{ij}:i\in[n]\}$.
Consider the following LP:
\begin{align}
    \LP_{\MMNL}=\max\quad
    &\sum_{j\in[m]}\nu_j\zeta_j
      \label{eq:lp_mmnl}\\
    \suchthat\quad
    &\zeta_j\le\theta+
      \sum_{i\in[n]}x_i
      \frac{\widehat v_{ji}}{1+\widehat v_{ji}}
      (\widehat r_{ij}-\theta)^+
      &&\forall j\in[m],\ \forall\theta\in\Theta_j,
      \notag\\
    &\sum_{i\in[n]}x_i\le\ell,
      \notag\\
    &\mathbf{x}\in[0,1]^n,
      \qquad\boldsymbol{\zeta}\in\RR_+^m.
      \notag
\end{align}
Here, $x_i$ represents the probability that product $i$
belongs to the initial assortment.
By Remark~\ref{remark:main_programs}, the constraints
indexed by $\Theta_j$ are equivalent to
\[
    \zeta_j\le
    \LP(\mathbf{x},\widehat{\mathbf{v}}_j,
                     \widehat{\mathbf{r}}_j).
\]
Thus, the objective maximizes a weighted sum of the
single-customer LP bounds from Section~\ref{section:main_result}.
We show below that $\LP_{\MMNL}$ is an upper bound on
$\OPT_{\MMNL}$.

LP~\eqref{eq:lp_mmnl} has polynomially many variables
and constraints.
Moreover, it has an optimal solution satisfying
$\sum_{i\in[n]}x_i=\ell$:
increasing the coordinates of $\mathbf{x}$ only relaxes
the class constraints, and $\ell\le n$.

To implement these marginal probabilities while respecting
the initial assortment limit, we use the linear version of
Srinivasan's sampling process~\citep{srinivasan2001distributions},
also known as pivotal sampling~\citep{deville1998unequal},
with the fixed product order $1,\ldots,n$.
Given $\mathbf{x}\in[0,1]^n$ with
$\sum_{i\in[n]}x_i=\ell$, this process samples an
assortment $S$ in polynomial time with distribution $\alpha$
satisfying
\[
    \PP_{S\sim\alpha}(i\in S)=x_i
    \quad(i\in[n]),
    \qquad
    |S|=\ell
    \quad\text{almost surely}.
\]
For this fixed linear order, the inclusion indicators are
negatively associated~\citep{dubhashi2007positive}.
These properties ensure that the initial assortment is
feasible and that Lemma~\ref{lem:threshold_lower_bound}
applies to each class.
Algorithm~\ref{alg:mmnl_customization} summarizes our solution.

\begin{algorithm}[H]
    \caption{Threshold policy for MMNL with customization}
    \label{alg:mmnl_customization}
    \begin{algorithmic}[1]
        \State Solve LP~\eqref{eq:lp_mmnl} and obtain an optimal
        solution $(\mathbf{x},\boldsymbol{\zeta})$ satisfying
        $\sum_{i\in[n]}x_i=\ell$.
        \For{each customer class $j\in[m]$}
            \State Compute the threshold $\tau_j\ge0$ satisfying
            \Statex \hspace{\algorithmicindent}
            $\displaystyle
                \tau_j=\sum_{i\in[n]}x_i
                \frac{\widehat v_{ji}}{1+\widehat v_{ji}}
                (\widehat r_{ij}-\tau_j)^+.
            $
            \State Set
            $H_j=\{i\in[n]:\widehat r_{ij}\ge\tau_j\}$.
        \EndFor
        \State Apply linear Srinivasan sampling with product
        order $1,\ldots,n$ and input $\mathbf{x}$ to draw
        an initial assortment $S$ of size $\ell$.
        \State Observe the customer's class $j\in[m]$.
        \State Display any subset of $S\cap H_j$ of size
        $\min\{|S\cap H_j|,k\}$.
    \end{algorithmic}
\end{algorithm}

\begin{proof}[Proof of Theorem~\ref{thm:mmnl}]
Let $\OPT_{\MMNL}$ be the optimum in
Problem~\eqref{eq:mmnl}, and let $\ALG_{\MMNL}$ be the
expected revenue of Algorithm~\ref{alg:mmnl_customization}.
We prove
\[
    \OPT_{\MMNL}\le\LP_{\MMNL}\le2\ALG_{\MMNL}.
\]

\emph{Upper bound on the customized optimum.}
Let $S^*$ be an optimal initial assortment for
Problem~\eqref{eq:mmnl}.
For each class $j$, deterministic availability of $S^*$
is feasible for the magician with marginal probability
bounds $\one_{S^*}$.
The magician can therefore obtain revenue
$R_{j,k}^+(S^*)$.
Lemma~\ref{lem:magician_upper_bound} gives
\[
    R_{j,k}^+(S^*)
    \le
    \LP(\one_{S^*},\widehat{\mathbf{v}}_j,
                      \widehat{\mathbf{r}}_j).
\]
By Remark~\ref{remark:main_programs}, this inequality
implies that
\[
    \mathbf{x}=\one_{S^*},
    \qquad
    \zeta_j=R_{j,k}^+(S^*)
    \quad(j\in[m])
\]
satisfy all class constraints in LP~\eqref{eq:lp_mmnl}.
The initial assortment constraint also holds because
$|S^*|\le\ell$.
This feasible solution has objective value
$\OPT_{\MMNL}$, proving
$\OPT_{\MMNL}\le\LP_{\MMNL}$.

\emph{Performance of the threshold policy.}
Take the optimal solution
$(\mathbf{x},\boldsymbol{\zeta})$ chosen in
Algorithm~\ref{alg:mmnl_customization}.
The sampled initial assortment has size $\ell$, and every
displayed assortment is a subset of it with size at most
$k$, so the policy is feasible.
Let $\alpha$ denote the distribution of the sampled
assortment $S$.
Its inclusion indicators are negatively associated and have
marginal probabilities $\mathbf{x}$.

For each class $j$, apply
Lemma~\ref{lem:threshold_lower_bound} with reference and
actual marginal probabilities both equal to $\mathbf{x}$,
preference weights $\mathbf{v}_j$, revenues
$\mathbf{r}_j$, and threshold set $H_j$.
Together with the class constraints in
LP~\eqref{eq:lp_mmnl}, this gives
\begin{align*}
    \zeta_j
    \le
    \LP(\mathbf{x},\widehat{\mathbf{v}}_j,
                     \widehat{\mathbf{r}}_j) \le 2\tau_j
    \le
    2\EE_{S\sim\alpha}
       \bigl[R_{j,k}^-(S\cap H_j)\bigr].
\end{align*}
Every assortment displayed by
Algorithm~\ref{alg:mmnl_customization} has size
$\min\{|S\cap H_j|,k\}$ and is contained in $S\cap H_j$.
Its expected revenue from a class-$j$ customer is therefore
at least $R_{j,k}^-(S\cap H_j)$.

The customer's class is independent of the sampling of $S$.
Weighting the class bounds by $\nu_j$ and summing yields
\[
    \LP_{\MMNL}
    =\sum_{j\in[m]}\nu_j\zeta_j
    \le
    2\EE_{S\sim\alpha}
      \left[
          \sum_{j\in[m]}\nu_jR_{j,k}^-(S\cap H_j)
      \right]
    \le2\ALG_{\MMNL}.
\]
Combining the two bounds proves the approximation guarantee.

Finally, LP~\eqref{eq:lp_mmnl} has polynomial size,
each threshold can be computed by sorting the distorted
revenues, and the sampling and display steps take
polynomial time.
Hence, Algorithm~\ref{alg:mmnl_customization} runs in
polynomial time.
\end{proof}

\section{Non-exclusive ride-hailing notifications}
\label{section:ride_hailing}

In non-exclusive ride-hailing notifications, a platform may
notify several drivers about the same rider before observing
their acceptances.
The notification set selection problem was introduced by
\citet{ekbatani2026non} to model such decisions at ride-hailing
platforms.

\paragraph{Notification set selection.}
The platform has a set of drivers $\cD$ and a set of riders
$\cR$.
It fixes all notifications before observing any acceptances
and may notify each driver about at most one rider.
When notified about rider $j\in\cR$, driver $i\in\cD$
accepts with probability $\rho_{ij}\in[0,1]$.
Conditional on the notifications, acceptances are independent
across drivers.
A match between driver $i$ and rider $j$ generates welfare
$r_{ij}\ge0$, and the platform seeks to maximize the expected
total welfare.

When several drivers accept a notification for the same rider,
a rule determines which driver is matched with that rider.
Under \emph{First Acceptance}, a driver is selected uniformly
from the accepting drivers.
Under \emph{Best Acceptance}, the platform selects an accepting
driver with the highest match welfare.
We additionally consider \emph{Worst Acceptance}, under which
an adversary selects an accepting driver with the lowest match
welfare.
If no driver accepts, the rider remains unmatched and
contributes zero welfare.
For every realization of the acceptances, the welfare under
Worst Acceptance is at most that under First Acceptance,
which is at most that under Best Acceptance.

\citet{ekbatani2026non} obtain a $1/4$-approximation guarantee
under First Acceptance, up to a multiplicative factor
$(1-\delta)$ and an additive loss $\mathcal{O}(\varepsilon)$.
Their algorithm runs in polynomial time for any fixed
$\delta>0$.
Applying the framework of Section~\ref{section:main_result},
we obtain a $1/2$ guarantee under Worst Acceptance against
the optimal expected welfare under Best Acceptance.
The same guarantee therefore also holds under First Acceptance.

\begin{theorem}\label{thm:notifications}
There is a randomized polynomial-time notification policy
whose expected welfare under Worst Acceptance is at least
half of the optimal expected welfare under Best Acceptance.
\end{theorem}

\paragraph{Our solution.}
We reduce the problem to constrained standard TSAO
(Section~\ref{section:standard_tsao}), using the ranked MNL
choice model introduced in
Section~\ref{subsection:ranked_mnl} and cardinality limit
$k=1$ on both sides.
Drivers play the role of customers, and riders play the role
of suppliers.
A driver's choice represents acceptance of a notification,
while a rider's response determines its matched driver.
The cardinality limit on the rider side restricts the final
displayed backlog to one driver; the number of drivers
notified about a rider remains unrestricted.
Under this reduction, Best Acceptance corresponds to
customized supplier responses, whereas Worst Acceptance
corresponds to adversarial truncation.
We apply Algorithm~\ref{alg:standard_tsao} with the ranked MNL
parameters and sampling rule described below.

\begin{proof}[Proof of Theorem~\ref{thm:notifications}]
Identify $\cD$ with the customers and $\cR$ with the
suppliers in standard TSAO, impose cardinality limit
$k=1$ on both sides, and use $r_{ij}$ as the match revenue.

For driver $i$, assign rider $j$ the rank and preference
weight
\[
    s_{ij}=\one\{\rho_{ij}=1\},
    \qquad
    v_{ij}=\frac{\rho_{ij}}{1-\rho_{ij}+s_{ij}}.
\]
Give the driver's exit option rank zero and preference
weight one.
Writing $\mathbf{s}_i=(s_{ij})_{j\in\cR}$ and
$\mathbf{v}_i=(v_{ij})_{j\in\cR}$, the ranked MNL
choice probability from a singleton satisfies
\[
    \phi_{\mathbf{v}_i,\mathbf{s}_i}(j,\{j\})
    =\rho_{ij}.
\]
This identity also holds when $\rho_{ij}=0$ or
$\rho_{ij}=1$.
Thus, displaying a singleton to a driver reproduces
the required notification acceptance probability.
The independent customer choices in standard TSAO reproduce
the independent driver acceptances.

For each rider, give every driver rank one and preference
weight one, and give the exit option rank zero and
preference weight one.
The rider then selects the driver in any displayed singleton
with probability one.
Customization selects an accepting driver with the largest
$r_{ij}$ and therefore implements Best Acceptance.
Adversarial truncation to one driver selects an accepting
driver with the smallest $r_{ij}$ and implements Worst
Acceptance.
Selecting the retained driver uniformly implements
First Acceptance.

The ranked MNL extension ensures that
Lemmas~\ref{lem:magician_upper_bound}
and~\ref{lem:threshold_lower_bound} apply to these choice
models.
To apply the proof of Theorem~\ref{thm:standard_tsao},
it remains to represent and implement the feasible
customer choice probabilities.
For each driver $i$, let $y_{ij}$ denote the probability
of sampling the singleton $\{j\}$.
The feasible choice probabilities have the following
linear representation:
\[
    y_{ij}\ge0
    \quad (i\in\cD,\ j\in\cR),
    \qquad
    \sum_{j\in\cR}y_{ij}\le1
    \quad (i\in\cD),
    \qquad
    x_{ij}=\rho_{ij}y_{ij}
    \quad (i\in\cD,\ j\in\cR).
\]
These constraints replace
$\mathbf{x}_i\in\cA_i$ in LP~\eqref{eq:lp_standard}.
They give an exact representation of the convex hull of
choice vectors generated by feasible notifications.

Given a feasible solution, independently for each driver
$i$, sample the singleton $\{j\}$ with probability $y_{ij}$
and the empty assortment with the remaining probability.
Then apply the threshold filtering step of
Algorithm~\ref{alg:standard_tsao}.
The resulting notifications are fixed before any acceptances
are observed, and each driver is notified about at most
one rider.

The proof of Theorem~\ref{thm:standard_tsao} therefore
applies: the resulting static notification policy has
expected welfare under Worst Acceptance at least half
of the adaptive customized optimum in the reduced instance.
This benchmark is at least the optimal expected welfare
of a static notification policy under Best Acceptance,
which proves the claimed guarantee.
The modified LP has polynomial size, and the threshold
computation and sampling steps take polynomial time.
\end{proof}

\section{Multi-round broadcast matching}
\label{section:broadcast_matching}

In multi-round broadcast matching, a firm offers orders to
drivers over several rounds and observes their responses
before choosing later offers.
The model was introduced by \citet{mohammadimulti}.

\paragraph{Multi-round broadcast model.}
Let $\cO$ be a set of orders, such as deliveries, let $\cD$
be a set of drivers, and let $T\in\NN$ be the number of rounds.
A match between order $i\in\cO$ and driver $j\in\cD$
generates revenue $r_{ij}\ge0$.
When offered order $i$, driver $j$ accepts with probability
$p_{ij}\in[0,1]$.
The acceptance variables are independent across
order--driver pairs.

In each round, each unmatched order may be offered to
several unmatched drivers, while each driver receives
at most one offer.
Each order--driver pair may be offered at most once during
the horizon.
If several drivers accept the same order in one round,
the order is assigned uniformly at random among them.
A matched order and driver leave the system; unmatched
orders and drivers remain available for later rounds.
The firm observes all earlier offers, acceptances, and
matches before choosing the next round's offers.
Let $\OPT_{\BM}$ denote the largest expected revenue of
a randomized adaptive policy.

For this problem, \citet{mohammadimulti} obtain approximation
factors of $(1-e^{-2})/2$ for the single-order setting and
$(1-e^{-3})/3$ for the general multiple-order setting.
Applying the framework of Section~\ref{section:main_result},
we obtain a $1/2$ guarantee for the general setting against
the same adaptive benchmark.

\begin{theorem}\label{thm:broadcast_matching}
For multi-round broadcast matching, there is a randomized
polynomial-time policy whose expected revenue is at least
half of the adaptive optimum $\OPT_{\BM}$.
\end{theorem}

\paragraph{Our solution.}
We apply the framework of Section~\ref{section:main_result}
separately to each order.
For a fixed order, view each driver--round pair as a product,
available when the driver's first acceptance occurs for
that order in that round.
The ranked MNL choice model introduced in
Section~\ref{subsection:ranked_mnl} represents the assignment
rule by giving earlier rounds higher ranks and equal
preference weights to drivers within a round.
We first construct an LP relaxation using distributions
over offer lists for each driver.
We then sample these lists independently and filter their
entries using a revenue threshold for each order.

Let $\cL$ be the set of lists
$L=(i_1,\ldots,i_T)\in(\cO\cup\{\emptyset\})^T$
in which each order appears at most once.
A list specifies the offers to one driver over the $T$ rounds.
An entry $i_t=\emptyset$ means that the driver receives
no offer in round $t$; this entry may be repeated.
Set $p_{\emptyset j}=0$ for every driver $j$.
For a driver $j\in\cD$ and an order $i=i_s$ appearing
in $L$, define
\[
    a_{ij}(L)
    =p_{ij}\prod_{h<s}(1-p_{i_hj}),
\]
and set $a_{ij}(L)=0$ when $i$ does not appear in $L$.
This is the probability that $i$ is the first order
accepted by driver $j$ when it considers the list in order
and stops after its first acceptance.

We use one variable $\lambda_{jL}$ for the probability
of assigning list $L$ to driver $j$:
\begin{align}
    \LP_{\BM}=\max\quad
    &\sum_{j\in\cD}\sum_{L\in\cL}
      \lambda_{jL}\sum_{i\in\cO}r_{ij}a_{ij}(L)
      \label{eq:broadcast_list_lp}\\
    \suchthat\quad
    &\sum_{j\in\cD}\sum_{L\in\cL}
      \lambda_{jL}a_{ij}(L)\le1
      &&\forall i\in\cO,
      \label{eq:broadcast_order_constraint}\\
    &\sum_{L\in\cL}\lambda_{jL}=1
      &&\forall j\in\cD,
      \label{eq:broadcast_driver_distribution}\\
    &\lambda_{jL}\ge0
      &&\forall j\in\cD,\ \forall L\in\cL.
      \notag
\end{align}
For each order,
Constraint~\eqref{eq:broadcast_order_constraint} limits
the sum of its marginal first-acceptance probabilities
to one.
Constraint~\eqref{eq:broadcast_driver_distribution}
specifies a probability distribution over lists for
each driver.
For every feasible solution $\boldsymbol{\lambda}$, write
\[
    y_{ij}
    =\sum_{L\in\cL}\lambda_{jL}a_{ij}(L)
\]
for the marginal probability that driver $j$ first accepts
order $i$ under its assigned list distribution.

The following lemma establishes both the validity and
tractability of this relaxation.

\begin{lemma}\label{lem:broadcast_list_lp}
The adaptive optimum satisfies
$\OPT_{\BM}\le\LP_{\BM}$.
Moreover, an optimal solution to
LP~\eqref{eq:broadcast_list_lp} with polynomially many
nonzero coefficients can be computed in polynomial time.
\end{lemma}

For the upper bound, we represent each driver's matching
probabilities under an adaptive policy by a distribution
over offer lists.
For tractability, separation in the dual reduces to
optimizing a single driver's list, which can be solved
by dynamic programming.
The proof is deferred to
Appendix~\ref{appendix:broadcast_list_lp}.

Given an optimal LP solution, compute for each order $i$
the threshold $\tau_i\ge0$ satisfying
\begin{equation}\label{eq:broadcast_threshold}
    \tau_i
    =\sum_{j\in\cD}y_{ij}(r_{ij}-\tau_i)^+.
\end{equation}
Independently sample a list for each driver, and replace
each entry $i$ with $r_{ij}<\tau_i$ by $\emptyset$.
This filtering preserves the round of every retained entry.

A driver becomes inactive for future offers at its first
acceptance, even if it loses the assignment lottery.
When a retained entry corresponds to an order that has
already matched, we simulate the driver's response using
acceptance probability $p_{ij}$.
A simulated acceptance also makes the driver inactive,
but creates no match.
These simulated responses make each driver's first
acceptance follow the same law as if it considered its
filtered list in isolation.
Algorithm~\ref{alg:broadcast_matching} summarizes the policy.

\begin{algorithm}[!htbp]
    \caption{Threshold policy for multi-round broadcast matching}
    \label{alg:broadcast_matching}
    \begin{algorithmic}[1]
        \State Solve LP~\eqref{eq:broadcast_list_lp} and obtain
        an optimal solution $\boldsymbol{\lambda}$ with
        polynomially many nonzero coefficients.
        \State Set
        $y_{ij}=\sum_{L\in\cL}\lambda_{jL}a_{ij}(L)$
        for every $i\in\cO$ and $j\in\cD$.
        \State For each order $i\in\cO$, compute the threshold
        $\tau_i$ satisfying
        Equation~\eqref{eq:broadcast_threshold}.
        \For{each driver $j\in\cD$}
            \State Independently sample $L_j\in\cL$ with
            probabilities $(\lambda_{jL})_{L\in\cL}$.
            \State Form $L_j'$ by replacing each entry $i$
            in $L_j$ with $r_{ij}<\tau_i$ by $\emptyset$.
        \EndFor
        \State Initialize the active drivers
        $\cD_{\mathrm{rem}}\gets\cD$ and mark every order
        unmatched.
        \For{$t=1,\ldots,T$}
            \For{each driver $j$ active at the start of round $t$}
                \State Let $i$ be the entry in position $t$
                of $L_j'$.
                \If{$i\ne\emptyset$}
                    \If{$i$ is unmatched}
                        \State Offer $i$ to $j$ and observe
                        the driver's response.
                    \Else
                        \State Independently simulate
                        $j$'s response to $i$, with acceptance
                        probability $p_{ij}$.
                    \EndIf
                    \If{$j$ accepts the real or simulated offer}
                        \State Remove $j$ from
                        $\cD_{\mathrm{rem}}$.
                    \EndIf
                \EndIf
            \EndFor
            \State Assign each unmatched order with at least
            one real acceptance in round $t$ uniformly among
            the drivers who accepted it.
        \EndFor
    \end{algorithmic}
\end{algorithm}

\begin{proof}[Proof of Theorem~\ref{thm:broadcast_matching}]
Let $\ALG_{\BM}$ denote the expected revenue of
Algorithm~\ref{alg:broadcast_matching}.
By Lemma~\ref{lem:broadcast_list_lp},
$\OPT_{\BM}\le\LP_{\BM}$.
It remains to prove $\LP_{\BM}\le2\ALG_{\BM}$.

The algorithm's real offers respect the round and pair
constraints.
Making a driver inactive only suppresses future offers,
so it preserves feasibility even when that driver has
not matched.
Only real acceptances generate matches, and every order
is assigned at most once.
Thus, the algorithm defines a feasible policy.

For each order $i$,
Constraint~\eqref{eq:broadcast_order_constraint} and
Equation~\eqref{eq:broadcast_threshold} give
\begin{align*}
    \sum_{j\in\cD}r_{ij}y_{ij}
    \le
    \tau_i\sum_{j\in\cD}y_{ij}
    +\sum_{j\in\cD}y_{ij}(r_{ij}-\tau_i)^+ \le
    \tau_i+\sum_{j\in\cD}y_{ij}(r_{ij}-\tau_i)^+
    =2\tau_i.
\end{align*}
Summing over orders yields
\begin{equation}\label{eq:broadcast_global_upper}
    \LP_{\BM}\le2\sum_{i\in\cO}\tau_i.
\end{equation}

We now show that the expected revenue from each order
$i$ is at least $\tau_i$.
For each driver $j$ and position $t$, define
\[
    x_{ijt}
    =\sum_{\substack{L\in\cL:\\ i_t=i}}
      \lambda_{jL}p_{ij}\prod_{s<t}(1-p_{i_sj}).
\]
This is the probability that $i$ is the first order
accepted by driver $j$ and appears in position $t$
of its unfiltered random list.
Since $\sum_{t=1}^T x_{ijt}=y_{ij}$,
Equation~\eqref{eq:broadcast_threshold} becomes
\begin{equation}\label{eq:broadcast_tau_copies}
    \tau_i
    =\sum_{j\in\cD}\sum_{t=1}^T
      x_{ijt}(r_{ij}-\tau_i)^+.
\end{equation}

Let $Z_{ijt}$ indicate that $i$ occupies position $t$
of the filtered list $L_j'$ and that driver $j$'s first
acceptance, real or simulated, occurs there.
Write $z_{ijt}=\PP(Z_{ijt}=1)$.
For a retained entry, filtering only removes earlier
competing offers.
Consequently,
$
    z_{ijt}\ge x_{ijt}
    \text{ whenever }r_{ij}\ge\tau_i.
$

Real and simulated responses can be coupled to independent
Bernoulli variables for the order--driver pairs.
Because a driver becomes inactive at its first acceptance
regardless of the assignment outcome, its first acceptance
depends only on its own sampled list and response variables.
The families $(Z_{ijt})_{i\in\cO,t\in[T]}$ are therefore
independent across drivers.
Within each such family, at most one indicator equals one.
Such a family is negatively associated, and negative
association is preserved under independent concatenation
and restriction to a subfamily~\citep{dubhashi1996balls}.
Thus, for each order $i$, the indicators
$(Z_{ijt})_{j\in\cD,t\in[T]}$ are negatively associated.

Fix an order $i$ and consider the ranked MNL model of
Section~\ref{subsection:ranked_mnl} with products
$(j,t)\in\cD\times[T]$.
Product $(j,t)$ has revenue $r_{ij}$, rank $T+1-t$,
and preference weight one.
The exit option has rank zero and weight one, and there
is no cardinality constraint.
Write $R_i$ for the resulting revenue function, and define
\[
    A_i=\{(j,t):Z_{ijt}=1\},
    \qquad
    H_i=\{(j,t):r_{ij}\ge\tau_i\}.
\]
By construction, $A_i\subseteq H_i$.

The ranked model chooses uniformly among the products
in the earliest represented round.
These acceptances are real: before the first acceptance
of order $i$, that order is still unmatched.
Hence, $R_i(A_i)$ is the firm's conditional expected revenue
from order $i$.

All products have positive rank, so their distorted revenues
are $r_{ij}$ and their coefficients
$\widehat v/(1+\widehat v)$ equal one.
Thus, Equation~\eqref{eq:broadcast_tau_copies} is precisely
the threshold equation for this ranked MNL instance.
The availability indicators are negatively associated,
and their marginal probabilities satisfy
$z_{ijt}\ge x_{ijt}$ on $H_i$.
Applying Lemma~\ref{lem:threshold_lower_bound} through
the extension in Section~\ref{subsection:ranked_mnl} gives
\[
    \EE[\text{revenue from order }i]
    =\EE[R_i(A_i)]
    \ge\tau_i.
\]
Summing over orders and using
Equation~\eqref{eq:broadcast_global_upper}, we obtain
$
    2\ALG_{\BM}
    \ge2\sum_{i\in\cO}\tau_i
    \ge\LP_{\BM}.
$
Together with Lemma~\ref{lem:broadcast_list_lp}, this proves
the approximation guarantee.

Finally, Lemma~\ref{lem:broadcast_list_lp} provides a
polynomial-time implementation of the LP step.
Each threshold can be computed by sorting the match
revenues, and the sampling, filtering, and broadcast
steps take polynomial time.
\end{proof}

\section{MNL prophet inequality}\label{section:mnl_prophet}

In the MNL prophet problem, introduced by \citet{goyal2023mnl},
a firm constructs an assortment as products arrive sequentially.
The firm learns each product's preference weight and revenue upon
arrival and must immediately decide whether to include it.
These decisions are irrevocable, and the final assortment is
subject to a cardinality constraint. The objective is to maximize
the expected revenue of the resulting assortment.

\paragraph{Model.}
Each product $i\in[n]$ has a random parameter pair $(v_i,r_i)$
drawn from a known distribution with finite support in
$\RR_+^2$, independently across products.
The preference weight and revenue of a given
product may be dependent.
When product $i$ arrives, the firm observes $(v_i,r_i)$ and
irrevocably accepts or rejects it.
The final assortment $T$ contains at most $k$ accepted products
and yields revenue
\[
    R_{\mathbf{r},\mathbf{v}}(T)
    =
    \frac{\sum_{i\in T}r_iv_i}
         {1+\sum_{i\in T}v_i}.
\]
We allow an almighty adversary to determine the arrival order:
the adversary observes all parameter realizations and all
randomness used by the firm's policy before choosing this order.

The usual prophet observes all parameter realizations before
choosing its best feasible assortment, with the parameter pairs
drawn independently across products.
Against this benchmark, \citet{goyal2023mnl} establish a tight
$1/2$ guarantee under a cardinality constraint, even against
an almighty adversary.

We consider a stronger benchmark, the magician.
Before the parameters are drawn, the magician may choose any
joint distribution of the parameter pairs that preserves the
prescribed distribution of each pair $(v_i,r_i)$.
After observing all realizations, it chooses its best feasible
assortment.
Write $\OPT_{\M}$ for its optimal expected revenue.
Since the independent joint distribution is a feasible choice
for the magician, $\OPT_{\M}$ is at least the usual prophet's
expected revenue.
The firm's parameter pairs remain independent across products;
the additional freedom to choose their dependence belongs
only to the benchmark.

Applying the framework of Section~\ref{section:main_result},
we obtain a static threshold policy that achieves a $1/2$
guarantee against the magician, even against an almighty
adversary. Thus, we retain the tight approximation factor
of \citet{goyal2023mnl} while strengthening the benchmark.

\begin{theorem}\label{thm:prophet_mnl}
Under a cardinality constraint, there exists a static threshold
policy whose expected revenue is at least $\OPT_{\M}/2$,
even against an almighty adversary.
\end{theorem}

\paragraph{Our solution.}
We apply the framework of Section~\ref{section:main_result}
by representing each possible realization of a product's
parameters as a separate product copy.
Exactly one copy of each original product is available.
This representation allows us to use the magician upper bound
and the threshold guarantee from Section~\ref{section:main_result}.

For each parameter pair $(v_i,r_i)$, let $\widehat v_i$
and $\widehat r_i$ be the distorted parameters defined in
Section~\ref{subsection:distorted_mnl_choice_model}, using
cardinality limit $k$.
Before any products arrive, the firm computes a nonnegative solution $\tau$ to
\begin{equation}\label{eq:prophet_threshold}
    \tau
    =
    \sum_{i\in[n]}
    \EE\left[
        \frac{\widehat v_i}{1+\widehat v_i}
        (\widehat r_i-\tau)^+
    \right].
\end{equation}
The expectation is taken over the distribution
of each product's parameter pair.
Since these distributions have finite support, the threshold
can be computed in polynomial time in their total support
size by sorting the distorted revenues and solving the
corresponding linear equation.

When product $i$ arrives, the firm accepts it if
$\widehat r_i\ge\tau$ and fewer than $k$ products have
already been accepted.
Thus, the policy uses a fixed threshold throughout the
arrival process.

\begin{proof}[Proof of Theorem~\ref{thm:prophet_mnl}]
Let $\ALG$ denote the expected revenue of the threshold policy.
For each product $i$, create one copy $(i,\omega)$ for every
parameter pair $\omega$ in its support, with the corresponding
preference weight and revenue.
Let $A$ be the random set containing the realized copy of
each original product, and let $\mathbf{x}$ be its vector
of marginal availability probabilities.

Within each product's family, exactly one copy is available,
so the availability indicators are negatively associated.
The families corresponding to different products are independent.
Consequently, the full vector of availability indicators is
negatively associated \citep{dubhashi1996balls}.

Every joint parameter distribution available to the magician
induces a distribution of copies in $\cP(\mathbf{x})$.
The magician in the copy instance may choose any distribution
in $\cP(\mathbf{x})$, so its value upper-bounds
$\OPT_{\M}$.
Writing $\widehat{\mathbf{v}}$ and $\widehat{\mathbf{r}}$
for the distorted parameters of the copies,
Lemma~\ref{lem:magician_upper_bound} gives
\[
    \OPT_{\M}
    \le
    \LP(\mathbf{x},\widehat{\mathbf{v}},
        \widehat{\mathbf{r}}).
\]

The threshold equation for the copy instance is exactly
Equation~\eqref{eq:prophet_threshold}.
Let $H$ be the set of copies whose distorted revenue is
at least $\tau$, and let
\[
    B=\{i\in[n]:\widehat r_i\ge\tau\}
\]
be the random set of eligible original products.
The available copies in $A\cap H$ correspond exactly
to the realized products in $B$.
Applying Lemma~\ref{lem:threshold_lower_bound} to the
copy instance therefore yields
\[
    \LP(\mathbf{x},\widehat{\mathbf{v}},
        \widehat{\mathbf{r}})
    \le
    2\tau
    \le
    2\EE[R_k^-(B)].
\]

For every realization and every arrival order, the policy
accepts exactly $\min\{|B|,k\}$ products from $B$.
By the definition of $R_k^-(B)$, its revenue is therefore
at least $R_k^-(B)$.
This pointwise comparison remains valid when the arrival
order is chosen by an almighty adversary.
Combining the bounds gives
\[
    \ALG
    \ge
    \EE[R_k^-(B)]
    \ge
    \tau
    \ge
    \frac12\OPT_{\M},
\]
which proves the theorem.
\end{proof}

\section{Conclusion}
We developed a unified framework for assortment optimization
and choice-based matching under the MNL model.
A simple threshold policy, analyzed through a compact LP
relaxation, achieves a tight $1/2$ guarantee against a benchmark
that may exploit both customization and control over dependence,
even under cardinality constraints with adversarial truncation.
Across our applications, this framework yields improved
approximation factors, extensions to more general settings,
and guarantees against stronger benchmarks through a common
analysis.

The applications studied here illustrate several ways to use
the framework. We hope researchers will find it useful as a
tool for other assortment and matching problems.
A natural starting point is to connect a problem-specific
LP relaxation to our single-customer model, allowing the
threshold guarantee to guide both algorithm design and
analysis. Exploring such connections may reveal further
settings in which simple policies achieve strong guarantees.

\bibliographystyle{plainnat}
\bibliography{bibliography}

\clearpage
\appendix
\section{Tight lower bounds for the pairwise gaps}\label{appendix:benchmark_gaps}
In this section, we prove tight lower bounds for the three pairwise gaps defined in Section~\ref{section:introduction}.
Theorem~\ref{thm:mnl} bounds all three gaps by $2$.
We provide a matching lower bound for each gap.

\subsection{Gap between the committed and delayed firms}
In this section, we prove that the gap between the committed and delayed firms is $2$, even without a cardinality constraint.
Recall that the gap between the committed and delayed firms is defined as the supremum of $\OPT_{\D}/\OPT_{\C}$ over all instances: it compares a firm that chooses its assortment after observing independent product availabilities with one that commits before observing them.
These values are
\[
    \OPT_{\D}=\EE_{A\sim\cI({\mathbf{x}})}[R_{\infty}^+(A)],
    \qquad
    \OPT_{\C}=\max_{S\subseteq[n]}\EE_{A\sim\cI({\mathbf{x}})}[R(S\cap A)].
\]
Fix $0<\varepsilon<1$ and take two products with no cardinality constraint, so $k=\infty$.
Set
\[
    {\mathbf{x}}=(1,\varepsilon),\qquad
    {\mathbf{v}}=(\varepsilon^{-1},1),\qquad
    {\mathbf{r}}=(1,2\varepsilon^{-1}).
\]
Product $1$ is always available, whereas product $2$ is available with probability $\varepsilon$.
The three nonempty assortments have revenues
\[
    R(\{1\})=\frac1{1+\varepsilon},\qquad
    R(\{2\})=\frac1\varepsilon,\qquad
    R(\{1,2\})=\frac3{1+2\varepsilon}.
\]
Committing to $\{1\}$ gives $1/(1+\varepsilon)$, while committing to $\{2\}$ gives $\varepsilon/\varepsilon=1$.
Committing to both products gives
\begin{align*}
    \EE_{A\sim\cI({\mathbf{x}})}[R(A)]
    &=\frac{1-\varepsilon}{1+\varepsilon}
      +\frac{3\varepsilon}{1+2\varepsilon}=1+\frac{\varepsilon(1-\varepsilon)}
      {(1+\varepsilon)(1+2\varepsilon)}
     >1.
\end{align*}
Thus the best committed assortment contains both products, and
\[
    \OPT_{\C}
    =1+\frac{\varepsilon(1-\varepsilon)}
      {(1+\varepsilon)(1+2\varepsilon)}.
\]
When product $2$ is available, its singleton revenue $1/\varepsilon$ exceeds both $1/(1+\varepsilon)$ and $3/(1+2\varepsilon)$, since $0<\varepsilon<1$.
The delayed firm therefore chooses product $2$ alone when it is available and product $1$ otherwise.
Its value is
\[
    \OPT_{\D}
    =\frac{1-\varepsilon}{1+\varepsilon}
      +\varepsilon\frac1\varepsilon
    =\frac2{1+\varepsilon}.
\]
Consequently,
\[
    \frac{\OPT_{\D}}{\OPT_{\C}}
    =\frac{2(1+2\varepsilon)}{1+4\varepsilon+\varepsilon^2}
    \longrightarrow2
    \qquad(\varepsilon\downarrow0).
\]
This proves the lower bound for the gap between the committed and delayed firms.
The gap between the committed firm and the magician also equals two, since $\OPT_{\M}\ge\OPT_{\D}$ and Theorem~\ref{thm:mnl} bounds this gap by $2$.

\subsection{Gap between the delayed firm and the magician}
The following construction is a variant of the
correlation-gap tightness example in
\citet[Appendix~D]{ahmadnejadsaein2025adaptive}.
In this section, we prove that the gap between the delayed firm and the magician is $2$ even without a cardinality constraint.
Recall that the gap between the delayed firm and the magician is defined as the supremum of $\OPT_{\M}/\OPT_{\D}$ over all instances: it compares the magician's best joint availability distribution with independent product availabilities, with the firm choosing its best feasible subset in both cases.
The two values are
\[
    \OPT_{\M}=\max_{\alpha\in\cP({\mathbf{x}})}\EE_{A\sim\alpha}[R_{\infty}^+(A)],
    \qquad
    \OPT_{\D}=\EE_{A\sim\cI({\mathbf{x}})}[R_{\infty}^+(A)].
\]

Take a large integer $m$, put $n=m+1$, and impose no cardinality constraint, so $k=\infty$.
For the first $m$ products and the last product, respectively, set
\[
    v_i=m^{-2},\quad r_i=m^{3/2},\quad x_i=m^{-1/2}
    \quad(i\in[m]),
    \qquad
    v_n=m,\quad r_n=1,\quad x_n=1.
\]
We call the first $m$ products the rare products.
Product $n$ is always available and has singleton revenue $m/(m+1)$.
An assortment of $q\in\{0,\ldots,m\}$ products from $[m]$ has revenue
\[
    \frac{q/\sqrt m}{1+q/m^2}.
\]
This expression increases with $q$, since, for $0\le q<m$,
\[
    \frac{(q+1)/\sqrt m}{1+(q+1)/m^2}
    -\frac{q/\sqrt m}{1+q/m^2}
    =\frac{1}{\sqrt m\,(1+q/m^2)(1+(q+1)/m^2)}>0.
\]
An assortment containing product $n$ and $q$ products from $[m]$ has revenue
\begin{align*}
    \frac{m+q/\sqrt m}{1+m+q/m^2}
    &=\frac{m}{1+m+q/m^2}\cdot1
      +\frac{1+q/m^2}{1+m+q/m^2}
       \frac{q/\sqrt m}{1+q/m^2}\le\max\left\{1,\frac{q/\sqrt m}{1+q/m^2}\right\}.
\end{align*}
The equality writes this revenue as a weighted average of $1$ and the revenue without product $n$; the two coefficients are positive and sum to one.
Thus, for a set $B$ containing product $n$ and $K=|B\cap[m]|$ other products,
\begin{equation}\label{eq:correlation_assortment_bound}
    R_{\infty}^+(B)
    \le\max\left\{1,\frac{K/\sqrt m}{1+K/m^2}\right\}
    \le\max\left\{1,\frac K{\sqrt m}\right\}.
\end{equation}
The first inequality follows by considering assortments with and without product $n$ and using the preceding increase in $q$.
Define the law $\alpha$ of the subsets of $[n]$ by
\[
    \alpha_{[n]}=m^{-1/2},\qquad
    \alpha_{\{n\}}=1-m^{-1/2},\qquad
    \alpha_S=0\quad\text{for all other }S\subseteq[n].
\]
All rare products are available together with probability $m^{-1/2}$, while product $n$ is always available.
The probabilities sum to one, and every inclusion marginal probability equals the corresponding coordinate of ${\mathbf{x}}$, so $\alpha\in\cP({\mathbf{x}})$.
When all products are available, offering $[m]$ earns $\sqrt m/(1+1/m)$.
This value exceeds $1$ for $m\ge4$.
When only product $n$ is available, the optimal revenue is $m/(m+1)$.
Hence
\begin{align*}
    \OPT_{\M}
    &\ge\EE_{A\sim\alpha}[R_{\infty}^+(A)]\ge(1-m^{-1/2})\frac{m}{m+1}
      +m^{-1/2}\frac{\sqrt m}{1+1/m}=\frac{m}{m+1}\bigl(2-m^{-1/2}\bigr).
\end{align*}
Under independent availability, draw $A\sim\cI(\mathbf{x})$ and set $K=|A\cap[m]|$.
We have
\[
    K\sim\operatorname{Bin}(m,m^{-1/2}),\qquad
    \EE[K]=\sqrt m,\qquad
    \EE\bigl[(K-\sqrt m)^2\bigr]
    =\sqrt m\,(1-m^{-1/2}).
\]
Therefore,
\begin{align*}
    \OPT_{\D}
    &\le\EE\left[\max\left\{1,\frac K{\sqrt m}\right\}\right]
     =1+\frac{\EE[(K-\sqrt m)^+]}{\sqrt m}\\
    &\le1+\frac{\EE|K-\sqrt m|}{\sqrt m}
     \le1+\sqrt{\frac{\EE[(K-\sqrt m)^2]}{m}}\\
    &=1+m^{-1/4}\sqrt{1-m^{-1/2}}
     \le1+m^{-1/4}.
\end{align*}
The first inequality uses Equation~\eqref{eq:correlation_assortment_bound}, the second uses $a^+\le|a|$ for $a\in\RR$, and the third uses the Cauchy--Schwarz inequality.
Finally, we get
\[
    \frac{\frac{m}{m+1}(2-m^{-1/2})}{1+m^{-1/4}}
    \le\frac{\OPT_{\M}}{\OPT_{\D}}
    \le\frac{\OPT_{\M}}{\OPT_{\C}}.
\]
The left side tends to $2$ as $m\to\infty$, so Theorem~\ref{thm:mnl} gives a gap of $2$ between the delayed firm and the magician.

\section{Proof of Lemma \ref{lem:standard_choice_polytope} }
\label{appendix:standard_choice_polytope}

Fix a customer $i\in\cC$.
Suppliers with $v_{ij}=0$ have choice probability zero under every assortment.
Removing these suppliers from an assortment preserves its choice-probability vector and its feasibility.
Their coordinates are also fixed to zero in $\cA_i$.
It therefore suffices to prove the claim after restricting to suppliers with positive preference weights.
For the remainder of the proof, assume $v_{ij}>0$ for every $j\in\cS$.

For a feasible assortment $A\subseteq\cS$, let
$x_{ij}(A)=\phi_{\mathbf{v}_i}(j,A)$.
For every $j\in A$,
\[
    \frac{x_{ij}(A)}{v_{ij}}
    +
    \sum_{h\in\cS}x_{ih}(A)
    =
    \frac{1+V_i(A)}{1+V_i(A)}
    =
    1.
\]
For $j\notin A$, the first term is zero and the sum is at most one, so the corresponding constraint also holds.
Furthermore,
\[
    \sum_{j\in\cS}\frac{x_{ij}(A)}{v_{ij}}
    +
    k\sum_{h\in\cS}x_{ih}(A)
    =
    \frac{|A|+kV_i(A)}{1+V_i(A)}
    \le
    k,
\]
where the inequality uses $|A|\le k$.
Thus, the choice-probability vector of every feasible assortment belongs to $\cA_i$.
Since $\cA_i$ is convex, it contains the convex hull of these vectors.

Conversely, take $\mathbf{x}_i\in\cA_i$ and set
$\delta_i=1-\sum_{h\in\cS}x_{ih}$.
Nonnegativity and the cardinality constraint imply
$\sum_{h\in\cS}x_{ih}\le1$, so $\delta_i\ge0$.
If $\delta_i=0$, the individual constraints force $x_{ij}=0$ for every $j$, contradicting $\sum_{h\in\cS}x_{ih}=1$.
Hence $\delta_i>0$.

The individual constraints and the cardinality constraint now give
\[
    0
    \le
    \frac{x_{ij}}{v_{ij}\delta_i}
    \le
    1
    \qquad(j\in\cS),
    \qquad
    \sum_{j\in\cS}
    \frac{x_{ij}}{v_{ij}\delta_i}
    \le
    k.
\]
Thus, the vector with coordinates $x_{ij}/(v_{ij}\delta_i)$ belongs to
\[
    \left\{
        \mathbf{y}\in[0,1]^{\cS}:
        \sum_{j\in\cS}y_j\le k
    \right\}.
\]
Every extreme point of this polytope is an incidence vector of an assortment of size at most $k$.
Indeed, a point with two fractional coordinates can be perturbed in both directions by increasing one coordinate and decreasing the other.
A point with exactly one fractional coordinate can be perturbed in that coordinate if the sum constraint is slack.
If the sum constraint is tight, exactly one fractional coordinate is impossible because $k$ is an integer.

Consequently, there are nonnegative coefficients
$(\lambda_A)_{A\subseteq\cS:\,|A|\le k}$ satisfying
\begin{equation}\label{eq:standard_decomposition}
    \frac{x_{ij}}{v_{ij}\delta_i}
    =
    \sum_{A\ni j}\lambda_A
    \qquad(j\in\cS),
    \qquad
    \sum_{A:\,|A|\le k}\lambda_A
    =
    1.
\end{equation}
All sums over $A$ below range over assortments $A\subseteq\cS$ with $|A|\le k$.
Define
\begin{equation}\label{eq:standard_assortment_law}
    \alpha_A
    =
    \lambda_A\delta_i\bigl(1+V_i(A)\bigr).
\end{equation}
This reweighting converts the decomposition of the scaled vector into a distribution over MNL assortment choice vectors.
The coefficients are nonnegative and sum to one, since
\begin{align*}
    \sum_A\alpha_A
    =
    \delta_i\sum_A\lambda_A
    +
    \delta_i\sum_A\lambda_A V_i(A)
    =
    \delta_i
    +
    \delta_i
    \sum_{j\in\cS}
    v_{ij}\sum_{A\ni j}\lambda_A
    =
    \delta_i+\sum_{j\in\cS}x_{ij}
    =
    1.
\end{align*}
The second equality uses $\sum_A\lambda_A=1$ and exchanges the order of summation.
The third uses Equation~\eqref{eq:standard_decomposition}, and the final equality follows from the definition of $\delta_i$.

Moreover, for every supplier $j$,
\begin{align*}
    \sum_A\alpha_A\phi_{\mathbf{v}_i}(j,A)
    =
    \sum_{A\ni j}
    \alpha_A\frac{v_{ij}}{1+V_i(A)}
    =
    \delta_i v_{ij}
    \sum_{A\ni j}\lambda_A
    =
    x_{ij}.
\end{align*}
The second equality uses Equation~\eqref{eq:standard_assortment_law}, and the third uses Equation~\eqref{eq:standard_decomposition}.
Thus, sampling $A$ with probabilities $(\alpha_A)$ generates the marginal choice probabilities $\mathbf{x}_i$.
This proves the convex-hull characterization.

To compute the coefficients $(\lambda_A)$, consider the linear feasibility problem in Equation~\eqref{eq:standard_decomposition}, with one nonnegative variable for each feasible assortment.
Its dual separation problem reduces to selecting at most $k$ suppliers with the largest positive dual coefficients.
This optimization can be performed by sorting.
The optimization--separation equivalence therefore gives a polynomial-time procedure for computing a feasible decomposition with polynomially many nonzero coefficients \citep{grotschel1984geometric}.
Applying Equation~\eqref{eq:standard_assortment_law} preserves this support and gives the required assortment distribution in polynomial time.

\section{Proof of Theorem~\ref{thm:fully_static}}
\label{appendix:fully_static_proof}
Fix an optimal deterministic fully static policy, which
exists because randomization over assortment profiles
cannot improve the optimum.
Let $\OPT_{\FS}$ denote the fully static optimum and $\ALG_{\FS}$ the expected revenue of Algorithm~\ref{alg:fully_static_tsao}.
We prove
\[
    \OPT_{\FS}\le\LP_{\FS}\le4\ALG_{\FS}.
\]

\emph{Upper bound on the fully static optimum.}
Let $S_i$ and $C_j$ denote the customer and supplier
assortments of this policy, respectively. 
Let $p_{ij}$ be the probability that customer $i$ chooses supplier $j$, and let $q_{ij}$ be the probability that supplier $j$ chooses customer $i$.
Define
\[
    x_{ij}
    =
    \one\{j\in S_i,\ i\in C_j\}
    \frac{(1+v_{ij})(1+w_{ji})}
         {(1+V_i(S_i))(1+W_j(C_j))}.
\]
Then $0\le x_{ij}\le1$ and $a_{ij}b_{ij}x_{ij}=p_{ij}q_{ij}$.
For each customer $i$, we have
\begin{align*}
    \sum_{j\in\cS}\widehat a_{ij}x_{ij}
    &=
    \sum_{j\in S_i}
    \one\{i\in C_j\}
    \frac{v_{ij}+1/k}{1+V_i(S_i)}
    \frac{1+w_{ji}}{1+W_j(C_j)}\le
    \sum_{j\in S_i}\frac{v_{ij}+1/k}{1+V_i(S_i)}
    =
    \frac{V_i(S_i)+|S_i|/k}{1+V_i(S_i)}
    \le1.
\end{align*}
The first inequality uses $w_{ji}\le W_j(C_j)$ for $i\in C_j$, and the second uses $|S_i|\le k$.
The same calculation on the supplier side gives
\[
    \sum_{i\in\cC}\widehat b_{ij}x_{ij}\le1
    \qquad(j\in\cS).
\]
Thus, $\mathbf{x}$ is feasible for LP~\eqref{eq:lp_fully_static}.
Since agents choose independently, its objective value is
\[
    \sum_{i\in\cC}\sum_{j\in\cS}r_{ij}a_{ij}b_{ij}x_{ij}
    =
    \sum_{i\in\cC}\sum_{j\in\cS}r_{ij}p_{ij}q_{ij}
    =\OPT_{\FS}.
\]
This proves $\OPT_{\FS}\le\LP_{\FS}$.

\emph{Customer-side guarantee.}
Take an optimal solution $\mathbf{x}$ of LP~\eqref{eq:lp_fully_static} and run Algorithm~\ref{alg:fully_static_tsao}.
For every pair $(i,j)$, let
\[
    c_{ij}
    =\phi_{\mathbf{v}_i}(j,S_i)
    =\one\{j\in S_i\}\frac{v_{ij}}{1+V_i(S_i)}.
\]
Conditional on $S_i$, this is the probability that customer $i$ chooses supplier $j$.
We first prove that
\[
    \LP_{\FS}
    \le
    2\EE\left[\sum_{i\in\cC}\sum_{j\in\cS}r_{ij}b_{ij}c_{ij}\right].
\]
Fix a customer $i$ and apply Lemma~\ref{lem:threshold_lower_bound} with reference marginal probabilities $\mathbf{x}_i$, preference weights $\mathbf{v}_i$, revenues $(r_{ij}b_{ij})_{j\in\cS}$, and cardinality limit $k$.
The customer constraint in LP~\eqref{eq:lp_fully_static} gives $\sum_{j\in\cS}x_{ij}\widehat a_{ij}\le1$.
Consequently, the individual bounds in LP~\eqref{eq:main_primal} already imply its total-choice constraint, and
\begin{align*}
    &\LP\left(
        \mathbf{x}_i,
        (\widehat v_{ij})_{j\in\cS},
        \left(\frac{kv_{ij}}{1+kv_{ij}}r_{ij}b_{ij}\right)_{j\in\cS}
    \right)=
    \sum_{j\in\cS}x_{ij}\widehat a_{ij}
    \frac{kv_{ij}}{1+kv_{ij}}r_{ij}b_{ij}
    =
    \sum_{j\in\cS}r_{ij}a_{ij}b_{ij}x_{ij}.
\end{align*}
The second equality uses $\widehat a_{ij}kv_{ij}/(1+kv_{ij})=a_{ij}$.

The threshold $\tau_i$ and eligible set $H_i$ in Algorithm~\ref{alg:fully_static_tsao} are those of Lemma~\ref{lem:threshold_lower_bound} for these parameters.
The set $A_i$ has independent inclusion indicators, so its distribution is negatively associated.
The assortment $S_i$ retains exactly $\min\{|A_i\cap H_i|,k\}$ suppliers and therefore earns at least the revenue under adversarial truncation.
Lemma~\ref{lem:threshold_lower_bound} yields
\[
    \sum_{j\in\cS}r_{ij}a_{ij}b_{ij}x_{ij}
    \le2\tau_i
    \le
    2\EE\left[\sum_{j\in\cS}r_{ij}b_{ij}c_{ij}\right].
\]
Summing over customers gives
\begin{equation}\label{eq:fully_static_customer_half}
    \LP_{\FS}
    \le
    2\EE\left[\sum_{i\in\cC}\sum_{j\in\cS}r_{ij}b_{ij}c_{ij}\right].
\end{equation}

\emph{Supplier-side guarantee.}
We next prove that
\[
    \EE\left[\sum_{i\in\cC}\sum_{j\in\cS}r_{ij}b_{ij}c_{ij}\right]
    \le2\ALG_{\FS}.
\]
Fix a supplier $j$.
For each customer $i$, fix $\widetilde r_{ij}\ge 0$ such that
\[
    \EE\left[r_{ij}c_{ij}\mid\one\{j\in S_i\}\right]
    =\one\{j\in S_i\}\widetilde r_{ij}.
\]
The customer assortments are independent, so the events $(j\in S_i)_{i\in\cC}$ are independent.
Moreover, $\PP(j\in S_i)\le x_{ij}$ because $S_i\subseteq A_i$.
The supplier constraint in LP~\eqref{eq:lp_fully_static} therefore implies
\[
    \sum_{i\in\cC}\PP(j\in S_i)\widehat b_{ij}
    \le
    \sum_{i\in\cC}x_{ij}\widehat b_{ij}
    \le1.
\]
Apply Lemma~\ref{lem:threshold_lower_bound} with marginal probabilities $(\PP(j\in S_i))_{i\in\cC}$, preference weights $\mathbf{w}_j$, revenues $(\widetilde r_{ij})_{i\in\cC}$, and cardinality limit $k$.
These independent availability events satisfy the negative-association assumption.
As on the customer side, the total-choice constraint of LP~\eqref{eq:main_primal} is redundant, so
\begin{align*}
    \LP\left(
        (\PP(j\in S_i))_{i\in\cC},
        \widehat{\mathbf{w}}_j,
        \left(\frac{kw_{ji}}{1+kw_{ji}}\widetilde r_{ij}\right)_{i\in\cC}
    \right)&=
    \sum_{i\in\cC}\PP(j\in S_i)\widehat b_{ij}
    \frac{kw_{ji}}{1+kw_{ji}}\widetilde r_{ij}
    \\&=
    \sum_{i\in\cC}b_{ij}\PP(j\in S_i)\widetilde r_{ij}
    =
    \sum_{i\in\cC}b_{ij}\EE[r_{ij}c_{ij}].
\end{align*}
The second equality uses $\widehat b_{ij}kw_{ji}/(1+kw_{ji})=b_{ij}$, and the third uses the definition of $\widetilde r_{ij}$.
Lemma~\ref{lem:threshold_lower_bound} therefore gives a supplier threshold policy, based only on the set $\{i:j\in S_i\}$, whose expected revenue with coefficients $(\widetilde r_{ij})_{i\in\cC}$ is at least half of this quantity.

For any supplier rule based only on $\{i:j\in S_i\}$, independence of the customer assortments gives
\[
    \EE\left[r_{ij}c_{ij}\mid\{h:j\in S_h\}\right]
    =\one\{j\in S_i\}\widetilde r_{ij}.
\]
Its conditional expected revenue therefore equals its revenue with coefficients $(\widetilde r_{ij})_{i\in\cC}$.
Algorithm~\ref{alg:fully_static_tsao} observes all sampled customer assortments and chooses the best feasible $C_j$ for the realized coefficients $(r_{ij}c_{ij})_{i\in\cC}$.
It earns at least as much in expectation as the supplier threshold policy.
Summing over suppliers yields
\[
    \EE\left[\sum_{i\in\cC}\sum_{j\in\cS}r_{ij}b_{ij}c_{ij}\right]
    \le2\ALG_{\FS}.
\]
Combining this inequality with Equation~\eqref{eq:fully_static_customer_half} gives $\LP_{\FS}\le4\ALG_{\FS}$.

Finally, all assortments in Algorithm~\ref{alg:fully_static_tsao} have size at most $k$ and are fixed before agents choose, so the policy is feasible and fully static.
LP~\eqref{eq:lp_fully_static} has polynomial size, and each customer threshold can be computed by sorting the distorted revenues.
Each supplier assortment solves a cardinality-constrained MNL revenue maximization problem, which can be solved in polynomial time~\citep{rusmevichientong2010dynamic}.
Hence, Algorithm~\ref{alg:fully_static_tsao} runs in polynomial time.

\section{Proof of Lemma~\ref{lem:broadcast_list_lp}}
\label{appendix:broadcast_list_lp}

We first show that LP~\eqref{eq:broadcast_list_lp}
upper bounds the adaptive optimum, and then establish
its polynomial-time solvability.

\vspace{-3mm}
\paragraph{Upper bound on the adaptive optimum.}
Fix a randomized adaptive policy $\pi$.
We may define $\pi$ arbitrarily on histories of probability
zero.
Implement the uniform assignment rule using independent
uniform priorities for the drivers in each round:
an order is assigned to its highest-priority accepting
driver.
The priorities are independent across rounds and independent
of the acceptance variables and the policy's randomness.

Fix a driver $j\in\cD$.
We construct a random list $L_j$ using a simulated
execution of $\pi$ over the $T$ rounds.
Whenever $j$ is offered order $i$ in round $t$, first
reveal the other drivers' acceptances and the priorities
for that round.
If $j$ would win order $i$ upon accepting, record $i$
in position $t$ of $L_j$ and continue the policy along
the history in which $j$ rejects the offer, without
revealing its acceptance variable for $(i,j)$.
Otherwise, record $\emptyset$ and reveal $j$'s response
as usual before continuing the policy.
Also record $\emptyset$ in any round in which $j$
receives no offer.
Since $\pi$ never offers the same pair twice,
$L_j\in\cL$.

The construction does not inspect $j$'s acceptance variables
for the orders retained in $L_j$.
Moreover, these variables cannot be inspected later,
because pairs are never offered twice.
Consequently, conditional on $L_j=L$, the acceptance
variables associated with the orders in $L$ remain
independent, with their original probabilities
$(p_{ij})_{i\in L}$.

Couple the actual execution of $\pi$ with this
simulated execution using the same policy randomness,
priorities, and acceptance variables.
The two histories coincide until $j$ accepts its first
retained order.
At that point, $j$ wins the order by construction and
leaves the system.
An acceptance of an order not retained in the list
cannot match $j$, since $j$ would lose that order's
assignment lottery.
It follows that
\[
    \PP_\pi
    \bigl(i\text{ is matched to }j\mid L_j=L\bigr)
    =a_{ij}(L).
\]

Set $\lambda_{jL}=\PP(L_j=L)$.
Then
\[
    \PP_\pi(i\text{ is matched to }j)
    =\sum_{L\in\cL}\lambda_{jL}a_{ij}(L)
    =y_{ij}.
\]
The construction is performed separately for each driver;
the resulting list distributions need not be independent.
For every driver $j$, their probabilities satisfy
$\sum_{L\in\cL}\lambda_{jL}=1$.
For every order $i$,
\[
    \sum_{j\in\cD}y_{ij}
    =\PP_\pi(i\text{ is matched})
    \le1.
\]
Thus, $\boldsymbol{\lambda}$ is feasible for
LP~\eqref{eq:broadcast_list_lp}.
Its objective value is
\[
    \sum_{i\in\cO}\sum_{j\in\cD}
    r_{ij}\PP_\pi(i\text{ is matched to }j)
    =\EE[\text{revenue of }\pi].
\]
Since $\pi$ was arbitrary, this proves
$\OPT_{\BM}\le\LP_{\BM}$.

\paragraph{Polynomial-time implementation.}
The dual of LP~\eqref{eq:broadcast_list_lp} is
\begin{align}
    \LP_{\BM}=\min\quad
    &\sum_{i\in\cO}\beta_i+
      \sum_{j\in\cD}\gamma_j
      \label{eq:broadcast_list_dual}\\
    \suchthat\quad
    &\gamma_j\ge
      \sum_{i\in\cO}(r_{ij}-\beta_i)a_{ij}(L)
      &&\forall j\in\cD,\ \forall L\in\cL,
      \label{eq:broadcast_list_dual_constraint}\\
    &\beta_i\ge0
      &&\forall i\in\cO,
      \notag\\
    &\boldsymbol{\gamma}\in\RR^{\cD}.
      \notag
\end{align}
The all-empty list forces $\gamma_j\ge0$ for every driver.
Fix a driver $j$ and a dual vector $\boldsymbol{\beta}$.
Separating
Constraint~\eqref{eq:broadcast_list_dual_constraint}
requires finding a list maximizing
\[
    \sum_{\substack{s\in[T]:\\i_s\ne\emptyset}}
    (r_{i_sj}-\beta_{i_s})p_{i_sj}
    \prod_{h<s}(1-p_{i_hj}).
\]
For any chosen set of orders, this expression is maximized
by placing them in decreasing order of $r_{ij}-\beta_i$.
Indeed, conditional on reaching two consecutive orders
$a$ and $b$, placing $a$ before $b$ rather than $b$
before $a$ changes the expected value by
\[
    p_{aj}p_{bj}
    \bigl[(r_{aj}-\beta_a)-(r_{bj}-\beta_b)\bigr].
\]
The probability of reaching the subsequent orders is
unchanged.
An adjacent-swap argument therefore establishes the
claimed ordering.
The positions of the $\emptyset$ entries do not affect
the objective value.

Reindex the orders so that
\[
    r_{1j}-\beta_1
    \ge\cdots\ge
    r_{|\cO|,j}-\beta_{|\cO|}.
\]
Let $V(s,d)$ be the largest value obtainable from orders
$s,\ldots,|\cO|$ using at most $d$ remaining list positions.
For $s\le|\cO|$ and $d\ge1$, we have
\begin{equation}\label{eq:broadcast_list_dp}
    V(s,d)
    =
    \max\left\{
        (r_{sj}-\beta_s)p_{sj}
        +(1-p_{sj})V(s+1,d-1),\
        V(s+1,d)
    \right\},
\end{equation}
with boundary conditions
\[
    V(|\cO|+1,d)=0,
    \qquad
    V(s,0)=0.
\]
The first term corresponds to including order $s$
before the remaining selected orders; the second
corresponds to omitting it.
After sorting, this dynamic program finds an optimal
list in time $\mathcal{O}(T|\cO|)$.
It therefore provides a polynomial-time separation
procedure for the dual.

The ellipsoid method and the optimization--separation
equivalence~\citep{grotschel1984geometric} solve
LP~\eqref{eq:broadcast_list_dual} and recover an optimal
solution to LP~\eqref{eq:broadcast_list_lp} with
polynomially many nonzero coefficients.
This proves the second claim.

\end{document}